\documentclass[twoside, reqno, 10pt]{amsart}

\usepackage[left=2.5cm, right=2.5cm, top=3cm, bottom=2.5cm]{geometry}
\usepackage[colorlinks=true,pdfstartview=FitV,linkcolor=blue,citecolor=blue,urlcolor=blue]{hyperref}

\usepackage[usenames]{xcolor}
\definecolor{labelkey}{rgb}{0,0,1}
\definecolor{Red}{rgb}{0.7,0,0.1}
\definecolor{Green}{rgb}{0,0.7,0}

\usepackage{amsfonts, amssymb, amsmath, amsthm, mathrsfs, bbm, cjhebrew, gensymb, textcomp, mathtools, dsfont, calligra, textcomp, textgreek}

\usepackage{commath, setspace, subcaption, parcolumns, multirow, accents, comment, marginnote, verbatim, empheq, enumerate, stackrel, enumitem, float, physics}

\usepackage[capitalize,nameinlink,noabbrev]{cleveref}

\usepackage{graphicx, graphics, epsfig, placeins, psfrag}

\numberwithin{equation}{section}

\usepackage{afterpage,cancel}

\newtheorem{Thm}{Theorem}[section]

\newtheorem{Prop}[Thm]{Proposition}
\newtheorem{Cor}[Thm]{Corollary}

\newtheorem{Rmk}[Thm]{Remark}

\newtheorem*{Thm*}{Theorem}

\usepackage[ruled,linesnumbered,vlined]{algorithm2e}

\usepackage{listings}

\definecolor{dkgreen}{rgb}{0,0.6,0}
\definecolor{gray}{rgb}{0.5,0.5,0.5}

\newcommand{\R}{\mathbb{R}}

\newcommand{\Obs}{\mathcal{O}}

\newcommand{\til}[1]{{\widetilde{#1}}}

\newcommand{\xt}{\til{x}}
\newcommand{\yt}{\til{y}}

\newcommand{\st}{\til{\s}}
\newcommand{\bt}{\til{\be}}
\newcommand{\rt}{\til{\rho}}

\newcommand{\be}{\beta}
\newcommand{\de}{\delta}
\newcommand{\De}{\Delta}
\newcommand{\gam}{\gamma}
\newcommand{\eps}{\epsilon}
\newcommand{\veps}{\varepsilon}

\newcommand{\s}{\sigma}

\newcommand\tx{\til{x}}
\newcommand\ty{\til{y}}
\newcommand\tz{\til{z}}

\usepackage{multicol}
\usepackage[normalem]{ulem}

\font\tenipa=tipa10
\def\schwa{{\tenipa\char64}}
 
\begin{document}

\title[Learning parameters of a chaotic system via partial observations]{Dynamically learning the parameters of a chaotic system using partial observations}

\author{E. Carlson, J. Hudson, A. Larios, V. R. Martinez, E. Ng, J. P. Whitehead}

\thanks{The authors would like to thank Daniel Grange for his helpful comments and suggestions.  The research of E.C. was supported in part by the NSF Graduate Research Fellowship Program under Grant No.1610400. Any opinions, findings, and conclusions or recommendations expressed in this material are those of the author(s) and do not necessarily reflect the views of the National Science Foundation.  Research conducted for this talk is supported in part by the Pacific Institute for the Mathematical Sciences (PIMS). The research and findings may not reflect those of the Institute.  E.C. acknowledges and respects the l\schwa $\text{k}^{\text{w}}$\schwa $\eta$\schwa n peoples on whose traditional territory the University of Victoria stands, and the Songhees, Esquimalt and \underline{W}S\'ANE\'C peoples whose historical relationships with the land continue to this day.  Author A.L. was partially supported by NSF grant CMMI-1953346. V.R.M. was partially supported by the PSC-CUNY Research Award Program under grant PSC-CUNY 64335-00 52.}

\date{August 18, 2021}

\begin{abstract}
Motivated by recent progress in data assimilation, we develop an algorithm to dynamically learn the parameters of a chaotic system from partial observations.  Under reasonable assumptions, we rigorously establish the convergence of this algorithm to the correct parameters when the system in question is the classic three-dimensional Lorenz system.  Computationally, we demonstrate the efficacy of this algorithm on the Lorenz system by recovering any proper subset of the three non-dimensional parameters of the system, so long as a corresponding subset of the state is observable.  We also provide computational evidence that this algorithm works well beyond the hypotheses required in the rigorous analysis, including in the presence of noisy observations, stochastic forcing, and the case where the observations are discrete and sparse in time.
\end{abstract}

\maketitle


{\noindent \small {\it {\bf Keywords: parameter estimation, data assimilation, Azouani-Olson-Titi (AOT) algorithm, nudging, Lorenz equations, inverse problems}
  }  
  }

\tableofcontents

\newpage
\section{Introduction}
\noindent
A fundamental concern when using any mathematical model is the need to precisely specify parameters that describe the physical situation in question. Although the fundamental physics that underlie these models are typically not a matter of debate, it is often difficult to precisely identify the various parameters that describe a particular setting of interest.  A topic of immense importance is then to accurately determine the parameters of a model given a limited set of observations of the system.  At the far extreme, one may try to identify fitting parameters that match the available data via a neural network or similar data-fitting algorithm. Although this works remarkably well in many settings, it is often not informative of the underlying physics of the problem (see \cite{ayed2019learning,foster2020learning,kutz2017deep,ma2018model,qian2020lift,trehan2017error} for just a few examples of this and similar approaches). This article is not a discussion on the phenomenological benefits of using data versus modeling via first principles, but we note that most modeling approaches fall somewhere along the spectrum between purely data-driven techniques such as standard neural networks, and models built on differential equations or other mathematical constructs that assume a deterministic evolution of the system. Just as the construction of such models falls on this spectrum, so does the identifying of parameters that specifically fit such a model. In this light, we note that the approach described below lies closer to the classical approach of modeling via first principles, but incorporates limited observational data as well in order to learn the parameters of the relevant dynamical model.

Purely data-driven approaches to parameter learning and estimation have their relevance and are useful in many contexts, but also have some drawbacks.  First, most machine learning approaches require a substantial amount of data to train on, which may not always be available for a given dynamical system of interest.  Second, these approaches are typically applied after the data is collected.  In other words, the data is observed and then a surrogate model (such as a neural network) is trained to model the data; the surrogate is then used for future predictions.  We will instead present an approach that relies on existing knowledge of the equations upon which the dynamical system evolves, but where various components of the system may not be known.  We present an algorithm that will learn the parts of the governing equations that are either unknown or uncertain, \textit{concurrently} as data is collected.  Colloquially, we call this `on-the-fly' learning of parameters.  Although we only present algorithms for certain types of observable data, the methods are very promising in their applicability to a variety of physically interesting systems, including the one-dimensional (1D) Kuramoto-Sivashinsky equation \cite{pachev2021concurrent}, two-dimensional (2D) incompressible Navier-Stokes equations \cite{Carlson_Hudson_Larios_2018}, and related systems.

The current investigation  is to study the accurate estimation of parameters based on restricted observational data. 
The most straight-forward such method would be to run a series of model simulations using all potential parameter combinations, as is done in, e.g., \cite{DiLeoni_Clark_Mazzino_Biferale_2018_inferring}, and then use linear regression to fit the estimated parameters to, e.g., the previously collected observational output. However, such brute force approaches can often be cost prohibitive.  
Other techniques have been proposed and tested in various settings (see \cite{baumeister1997line,ding2019gradient,nguyen2016state,raue2015data2dynamics,ruthotto2017jinv,xu2015application,xun2013parameter} for a small number of examples).  These approaches range from Bayesian-based methods such as Markov Chain Monte Carlo (see, e.g., \cite{dashti2017bayesian}) --- which require a significant number of forward simulations of the model system in order to collect comparative data relative to the observations --- to various modifications of the Kalman filter suited for parameter learning/estimation (see, e.g., \cite{evensen2009ensemble,van2001square}), and even to particle filters (see, e.g.,  \cite{zhu2019state}).  The Kalman filter approach is most similar to our current investigation in practice because it is capable of reproducing the state of the system and the relevant parameters simultaneously with a limited number of observations and evaluations of the forward model.  Nevertheless, the algorithm presented below is quite different from the Kalman filter.

Variations on the parameter learning algorithm presented here are also demonstrated in a different context in \cite{Carlson_Hudson_Larios_2018} and \cite{pachev2021concurrent}.  All of these results, including the current investigation, are motivated by the continuous data assimilation approach pioneered by \cite{azouani2014continuous}. The approach was originally presented for the 2D Navier-Stokes equations but has since been generalized to settings beyond the Navier-Stokes system where only a subset of the prognostic variables are observable 
\cite{albanez2016continuous,
Altaf_Titi_Knio_Zhao_Mc_Cabe_Hoteit_2015,
biswas2018continuous,
Biswas_Bradshaw_Jolly_2020,
Biswas_Foias_Mondaini_Titi_2018downscaling,
Biswas_Martinez_2017,
Biswas_Price_2020_AOT3D,
Carlson_Larios_2021_sens,
Celik_Olson_Titi_2019,
Chen_Li_Lunasin_2021,
Desamsetti_Dasari_Langodan_Knio_Hoteit_Titi_2019_WRF,
di2020synchronization,
Diegel_Rebholz_2021,
farhat2020data,
farhat2016charney,
Farhat_Jolly_Titi_2015,
Farhat_Lunasin_Titi_2016abridged,
Farhat_Lunasin_Titi_2016benard,
Farhat_Lunasin_Titi_2017_Horizontal,
Franz_Larios_Victor_2021,
GarciaArchilla_Novo_2020,
GarciaArchilla_Novo_Titi_2018,
Gardner_Larios_Rebholz_Vargun_Zerfas_2020_VVDA,
Gesho_Olson_Titi_2015,
Jolly_Martinez_Olson_Titi_2018_blurred_SQG,
Jolly_Martinez_Titi_2017,
Larios_Pei_2017_KSE_DA_NL,
Larios_Rebholz_Zerfas_2018,
Larios_Victor_2019,
Lunasin_Titi_2015,
Markowich_Titi_Trabelsi_2016_Darcy,
Mondaini_Titi_2018_SIAM_NA,
Pei_2019,
Rebholz_Zerfas_2018_alg_nudge,
Zerfas_Rebholz_Schneier_Iliescu_2019}.
Further extensions of this approach to discrete-in-time observations \cite{Foias_Mondaini_Titi_2016,ibdah2020fully,Larios_Victor_2021_chiVsdelta2D}, and the presence of stochastic noise in the observations \cite{bessaih2015continuous,blocher2018data} have been made with completely rigorous justification.  The basic premise to this form of data assimilation is to include a feedback control term that ``nudges'' the modeled system toward the observed true state (see \cref{sect:multi}).  Dissipation is then crucially used to prove that the system converges not only to the observed projection of the true state, but in fact to the full true state.  It was noted in \cite{Carlson_Hudson_Larios_2018,farhat2020data,Larios_Pei_2018_NSV_DA} for different systems that if the parameters of the model are not that of the true state, then the system will converge to a finite amount of unrecoverable error.  This remaining error is a direct consequence of the parameter error, and hence provides an avenue to estimating the true value of the parameter.

As illustrated in \cite{Carlson_Hudson_Larios_2018,pachev2021concurrent} and below, these parameter learning methods are robust across a variety of settings.  Computational evidence is given in \cite{Carlson_Hudson_Larios_2018} that the viscosity of the 2D Navier-Stokes system can be recovered, and \cite{pachev2021concurrent} accurately identifies several different parameters simultaneously for the 1D Kuramoto-Sivashinsky equation.  However, neither of the works \cite{Carlson_Hudson_Larios_2018,pachev2021concurrent} were able to provide a completely rigorous justification for the success of these methods, despite some clear algorithmic hints that such a proof is available.  In order to focus on a setting where rigorous results are more readily achievable, we study the classical set of ordinary differential equations given by the Lorenz `63 system from \cite{lorenz1963deterministic}.  Indeed, in the context of the Lorenz equations, it has been rigorously proven that the feedback control approach results in synchronization of the modeled system with the true state. This has been accomplished in various setups, such as time-averaged observations, enforcing a nonlinear feedback control, or in the presence of observational error \cite{blocher2018data, Du_Shiue2021,LawShuklaStuart_2014}.  While the Lorenz system is inherently finite-dimensional and hence of a much simpler nature than even the Kuramoto-Sivashinsky equation, it is highly nonlinear and has often been used as the basis for investigations of nonlinear and chaotic phenomena such as turbulence (see, e.g., \cite{agarwal2016maximal,doering1995shape,foias2001lorenz,hayden2011discrete,souza2015maximal}).  Just as in \cite{foias2001lorenz,souza2015maximal}, we anticipate that the analysis outlined below can later be extended to infinite dimensional dissipative systems such as the Navier-Stokes or Kuramoto-Sivashinsky equations

The remainder of this article proceeds as follows.  Section \ref{sect:multi} provides a heuristic derivation of the update formula, then reviews some key prior results for the Lorenz system, and finally states the main mathematically rigorous results as theorems, the proofs of which are relegated to \cref{sect:proofs}. Afterwards, an alternative approach for parameter recovery via `direct-replacement' is provided for the sake of an analytical comparison. Section \ref{sec:comps} describes the computational experiments that were performed on the Lorenz system. These results demonstrate our parameter learning algorithms under various circumstances that are substantially more robust than the convergence regimes identified in the theorems stated in Section \ref{sect:multi} (see, for instance, \eqref{cond:mu:large} and \eqref{cond:non:degen}). The authors have made the code for this work freely available on Github (see \cref{sec:numerical_methods}), but for quick reference, a compressed MATLAB code is given in Appendix \ref{sect:matlab}. Lastly, Section \ref{sec:conclusions} includes some conclusions and discussion of future work.

\section{Multi-parameter recovery}\label{sect:multi}

In this section, we derive formulas that recover any subset of the parameters of the Lorenz `63 system provided that a certain subset of its state variables are known (\cref{sect:derive}). These formulas are ultimately inspired by the recent work \cite{Carlson_Hudson_Larios_2018}, which leverages a data assimilation algorithm for PDEs, developed by Azouani, Olson, and Titi (AOT) in \cite{azouani2014continuous}, to recover the unknown viscosity from partial observations of the velocity field. The AOT algorithm modifies the original PDE using a feedback control term that incorporates observations of the original, ``true'' system. More specifically, it appropriately interpolates the observations to the phase space of the original system in such a way that the state variables of the resulting model system are driven towards the observed variables of the original system.  Dissipative effects in the systems then drive the simulated solution to the true solution. For many of the systems studied in hydrodynamics, the solution of the AOT system asymptotically synchronizes with the solution of the original system when all of the system parameters are perfectly known 
\cite{
Altaf_Titi_Knio_Zhao_Mc_Cabe_Hoteit_2015,
biswas2018continuous,
Biswas_Martinez_2017,
Biswas_Price_2020_AOT3D,
Desamsetti_Dasari_Langodan_Knio_Hoteit_Titi_2019_WRF,
di2020synchronization,
farhat2020data,
farhat2016charney,
Farhat_Jolly_Titi_2015,
Farhat_Lunasin_Titi_2016abridged,
Farhat_Lunasin_Titi_2016benard,
Farhat_Lunasin_Titi_2017_Horizontal,
Jolly_Martinez_Olson_Titi_2018_blurred_SQG,
Jolly_Martinez_Titi_2017,
Markowich_Titi_Trabelsi_2016_Darcy,
Pei_2019}.  
In the case where the parameters are \textit{not} exactly known, however, we extract a formula from the AOT system that dynamically updates the unknown parameters {\sout in a systematic manner}. We further identify rigorous conditions under which these updates eventually converge to the true value of the parameter.

In the context of finite-dimensional systems, the AOT algorithm coincides with a classical data assimilation algorithm known as \textit{nudging} \cite{hoke1976initialization}. The difference between the AOT and nudging algorithms at the level of PDEs has recently been studied in detail for a large-scale comparison test case  \cite{Carlson_VanRoekel_Petersen_Godinez_Larios_2021}. For the Lorenz system, the AOT algorithm can be reduced to the classical nudging-based algorithm by selecting the observation-interpolating operator to be a diagonal matrix with non-negative coefficients---the coefficients representing the strength of nudging in each component of the observations. These coefficients ultimately determine the rate at which the model variables relax to the true variables. We will prove that the update rules derived from the AOT algorithm eventually recover the parameters, assuming certain algorithmic stipulations are met. We derive the formulas in \cref{sect:derive} and supply the proofs
in \cref{sect:proofs}.  A precise statement of the main results are presented in \cref{sect:statements}).
We conclude the section by briefly comparing the nudging-based approach with an alternative direct-replacement approach in \cref{sect:direct:replace}.

\subsection{Derivation of parameter recovery formulas}\label{sect:derive}
The system of interest is the Lorenz `63 system \cite{lorenz1963deterministic}, which is given by

\begin{empheq}[left=\empheqlbrace]{align}\label{eq:Lor}\begin{split}
       \dot{x} &= \sigma(y-x) \\
       \dot{y} &= \rho x - y - xz \\
       \dot{z} &= xy - \beta z
\end{split}\end{empheq}
where $\s,\rho,\be>0$. Consider the problem where a subset of the system parameters $\{\s,\rho,\be\}$ is unknown and we want to learn the true values of those unknown parameters by leveraging the \textit{known} dynamics, i.e., \eqref{eq:Lor}, along with partial observations of the state variables $x,y,z$ and estimates of the true parameter values.

The nudged system corresponding to \eqref{eq:Lor} is given by
\begin{empheq}[left=\empheqlbrace]{align}\label{eq:Lor:ng}\begin{split}
       \dot{\tx} &= \st(\ty-\tx) - \mu_1(\tx - x) \\
       \dot{\ty} &= \rt\tx - \ty - \tx\tz - \mu_2(\ty-y) \\
       \dot{\tz} &= \tx\ty - \bt\tz - \mu_3(\tz-z).
\end{split}\end{empheq}
where $\st,\rt,\bt>0$ and $\mu_1,\mu_2,\mu_3\geq0$. The nudged system \eqref{eq:Lor:ng} possesses the following synchronization property when the parameters are known exactly, i.e., $\s=\st$, $\rho=\rt$, and $\be=\bt$: for $\mu_1,\mu_2,\mu_3$ sufficiently large, one has $\tx-x$, $\ty-y$, $\tz-z\rightarrow0$ exponentially fast as $t\rightarrow\infty$. The reader is referred to, e.g., \cite{blocher2018data}, for additional details.

The parameter recovery problem we consider for \eqref{eq:Lor} is stated as follows: suppose that a subset of the system parameters, $\s,\rho,\be$, is unknown; the goal is to infer the unknown parameters, assuming that a subset of the state variables is observed. For example, if $\s$ is not known precisely but the continuous time series $\Obs=\{x(t)\}_{t\geq0}$ is observable, then the goal is to recover $\s$ using the values from $\mathcal{O}$ and the nudged system. This will be done by using observations to improve the accuracy of the nudged system and generating approximations concurrently. Based on the synchronization that occurs for \eqref{eq:Lor:ng}, we define an update formula that proposes increasingly accurate estimates of the unknown parameter(s) of interest.

Denote the differences between the nudged variables and original variables by
    \begin{align}\label{def:diff}
        u:=\tx-x,\quad v:=\ty-y,\quad w:=\tz-z,
    \end{align}
and the parameter errors by
\begin{align}\label{def:par:err}
    \Delta\sigma := \st-\sigma,\quad \Delta\rho := \rt-\rho,\quad \Delta\beta := \bt-\beta.
\end{align}
Then the system governing the evolution of the error vector $(u,v,w)$ is given by
\begin{align}\label{eq:diff_pre}
    \begin{cases}
        \dot{u} &=  (\ty-\tx)\Delta\sigma+\sigma v-(\mu_1+\sigma) u,\\
        \dot{v} &= \De\rho\tx +\rho u -uz-\tx w-(1+\mu_2)v, \\
        \dot{w} &= -\De\be\tz  +uy+\tx v-(\mu_3+\beta)w.
    \end{cases}
\end{align}
Multiplying each equation in \eqref{eq:diff_pre} by $u, v,$ and $w$ respectively, we obtain
\begin{empheq}[left=\empheqlbrace]{align*}
\begin{split}
    \frac{1}{2}\frac{d}{dt}u^2 &=  (\Delta\sigma)(\ty-\tx) u + \sigma(v-u)u  - \mu_1 u^2, \\
    \frac{1}{2}\frac{d}{dt}v^2 &= \De\rho \tx v + \rho uv - v^2 - (uz + \tx w)v - \mu_2v^2, \\
    \frac{1}{2}\frac{d}{dt}w^2 & = (uy + \tx v)w - \beta w^2 - \De\beta \tz w - \mu_3w^2.
\end{split}
\end{empheq}

Now suppose that $\mu_1,\mu_2,\mu_3\gg1$.  From the synchronization property, we anticipate that the parameter errors, $\Delta\s, \Delta\rho, \Delta\be$, are small and hence $u,v$ and $w$ should also be small. In particular, quadratic terms should be negligible \textit{unless} they are multiplied by $\mu_1,\mu_2,\mu_3$. Provided that the parameter errors are not identically zero, the differences are expected to relax towards a value proportional to the size of these errors. This means that the associated time-derivatives of the differences will also become negligible after a transient time interval. To leading order, we deduce that
\begin{empheq}[left=\empheqlbrace]{align*}
\begin{split}
    0 &\approx \De\sigma(\ty-\tx) u - \mu_1 u^2, \\
    0 &\approx \De\rho \tx v - \mu_2v^2, \\
    0 &\approx -\De\beta\tz w - \mu_3w^2,
\end{split}
\end{empheq}
which formally reduces to
\begin{empheq}[left=\empheqlbrace]{align*}
\begin{split}
    \sigma &\approx \st - \mu_1 \frac{u}{\ty-\tx}, \\
    \rho &\approx \rt - \mu_2\frac{v}{\tx}, \\
    \beta &\approx \bt + \mu_3\frac{w}{\tz}.
\end{split}
\end{empheq}
In particular, we propose the following rules for updating the parameters
\begin{empheq}[left=\empheqlbrace]{align}
\label{def:update}
    \begin{split}
    \sigma_{n+1} &= \sigma_n - \mu_1 \frac{u}{\ty-\tx}, \\
    \rho_{n+1} &= \rho_n - \mu_2\frac{v}{\tx}, \\
    \beta_{n+1} &= \beta_n + \mu_3\frac{w}{\tz},
    \end{split}
\end{empheq}
where it is understood that $\tx, \ty, \tz$ and $u, v, w$ are evaluated at the ``final time" of the previous ``epoch". This time delineates the period between the $(n-1)$-{st} and $n$-th updates. These parameter updates are only performed at instances when the differences $u,v,w$ have relaxed to a (non-zero) near constant value. The process is restarted afterwards by solving the nudging equations forward-in-time with the new parameter values. As this process is iterated, we find that the parameters converge to the true values.

\subsection{Statements of main theorems}\label{sect:statements}
In this section, we identify sufficient conditions under which the formulas \eqref{def:update} converge to the true values of the corresponding parameters. A rigorous proof of this statement is supplied in \cref{sect:proofs}, but we provide precise statements of these theorems now. In order to do so, we will make reference to the following well-known properties of the Lorenz system, which can be found in, e.g., \cite{Doering_Gibbon_1995_book,RobinsonBook}.

\begin{Thm}\label{thm:Lor:ab}
Let $\s,\rho,\be>0$. Then there exists an absorbing ball for \eqref{eq:Lor}, that is there exists $R>0$ such that for any $r>0$
    \begin{align}\label{eq:Lor:ab}
        |(x(t),y(t),z(t)-\rho-\s)|\leq R,\quad \text{for all}\ t\geq T,
    \end{align}
whenever $|(x(0),y(0),z(0)-\rho-\s)|\leq r$, for some $T=T(r)$. In particular, \eqref{eq:Lor} has a global attractor, $\mathscr{A}=\mathscr{A}(\s,\rho,\be)$.
\end{Thm}

In the following, we denote the ball of radius $r>0$ centered at $\mathbf{x}\in\R^3$ by $\mathscr{B}(r,\mathbf{x})$. Then the absorbing ball of \eqref{eq:Lor} may be denoted by $\mathscr{B}(R,0,0,-\rho-\sigma)$, where $R>0$ is given as in \cref{thm:Lor:ab}. For convenience, we will often simply write $\mathscr{B}_{\s,\rho,\be}$ to denote the absorbing ball of \eqref{eq:Lor} corresponding to parameters $\s,\rho,\be$,  or simply $\mathscr{B}$ whenever $\s,\rho,\be$ is understood to have been fixed.

Given $(x_0,y_0,z_0-\rho-\s)\in\mathscr{B}$, we denote the corresponding global solution of \eqref{eq:Lor} evaluated at time $t\geq0$ by $(x(t;x_0),y(t;y_0),z(t;z_0))$. For convenience, we suppress the dependence on the initial data and simply write $(x(t),y(t),z(t))$. Given $n>0$ and $t_n>0$, we use the notation
    \begin{align}\label{def:discrete}
        x_n=x(t_n),\quad y_n=y(t_n),\quad z_n=z(t_n).
    \end{align}
For a sequence of discrete times $\{t_n\}_{n\geq0}$,  let $I_n=[t_n,t_{n+1})$ where $t_0=0$. Then given positive sequences of the parameters $\{\s_n\},\{\rho_n\},\{\be_n\}$ updated as described above, we consider the unique solution $(\tx^{(n)},\ty^{(n)},\tz^{(n)})$ of \eqref{eq:Lor:ng} over the interval $I_n$ corresponding to initial data $(\tx(t_n),\ty(t_n),\tz(t_n))$ and parameters $\st=\s_n$, $\rt=\rho_n$, $\bt=\be_n$, i.e.,
    \begin{align}\label{eq:Lor:ng:n}
        \begin{cases}
    \dot{\tx}^{(n)} = \s_n(\ty^{(n)}-\tx^{(n)}) - \mu_1(\tx^{(n)} - x),&  \tx^{(n)}(t_{n}^-)=\tx_n, \\
       \dot{\ty}^{(n)} = \rho_n\tx^{(n)} - \ty^{(n)} - \tx^{(n)}\tz^{(n)} - \mu_2(\ty^{(n)}-y),& \ty^{(n)}(t_n^-)=\ty_n, \\
       \dot{\tz}^{(n)} = \tx^{(n)}\ty^{(n)} - \be_n\tz^{(n)} - \mu_3(\tz^{(n)}-z),& \tz^{(n)}(t_n^-)=\tz_n,
      \end{cases}
    \end{align}
for $t\in I_n$ and $n\geq0$, where $\mu_1,\mu_2,\mu_3\geq0$ are fixed. Note that when $n=0$, we simply let $t_n^-\equiv0$. We also have the corresponding systems \eqref{eq:diff} (evolution of the differences $u^{(n)},~v^{(n)}$ and $w^{(n)}$) and \eqref{eq:diff:dot} (evolution of $\dot{u}^{(n)},~\dot{v}^{(n)},$ and $\dot{w}^{(n)}$, denoted $\gam^{(n)}$, $\de^{(n)}$, and $\eta^{(n)}$, respectively) defined over $I_n$. Since we will be making parameter updates sequentially in time, we emphasize that the ``final values" of \eqref{eq:Lor:ng:n} over $I_{n-1}$ specify the initial value over the ``current interval" $I_n$,  while the initial value problem over $I_n$ for the \textit{derivative} system of \eqref{eq:Lor:ng:n} (see \eqref{eq:diff:dot}) is accordingly defined by using right-hand limits at the endpoint $t=t_n$.  This detail is necessary to allow for jumps in the solution at each instance of an update. In particular
    \begin{gather}\label{def:discrete:diff}
        \begin{split}
        \tx_{n}=\tx^{(n)}(t_{n}^-),\quad \ty_{n}=\ty^{(n)}(t_{n}^-),\quad \tz_{n}=\ty^{(n)}(t_{n}^-),\\
        u_{n}=u^{(n)}(t_{n}^-),\quad v_{n}=v^{(n)}(t_{n}^-),\quad w_{n}=w^{(n)}(t_{n}^-),\\
        \gam_{n}=\gam^{(n)}(t_{n}^+),\quad \de_{n}=\de^{(n)}(t_{n}^+),\quad \eta_{n}=\eta^{(n)}(t_{n}^+).
        \end{split}
    \end{gather}

\noindent
Finally, for $n\geq1$, we will suppose that $t_n>0$ represents the time at which the $n$-th update for the triplet $(\s,\rho,\be)$ is made. We then claim the following.
\begin{Thm}\label{thm:convergence:x}
Let $\rt=\rho$, $\bt = \beta$, and $\mu_2=\mu_3 = 0$.  Let $M \geq 1$ and $\s_0>0$ such that $|\sigma_0-\sigma| \leq M$, and choose a tolerance $\veps > 0$.  There exist constants $C,C'>0$ such that if $\mu_1$ satisfies
    \begin{align}\label{cond:mu:large}
        \mu_1\geq C\left(1+\s+\rho+\be+R\right)^2,
    \end{align}
and
    \begin{align}\label{cond:mu:converge}
        \mu_1\geq\frac{C'}{\veps^2},
    \end{align}
then if there exists a sequence of $N$ update times $0<t_1<t_2<\dots<t_N$ where 
    \begin{align}\label{cond:non:degen}
      \inf_{0\leq n\leq N+1}|\ty_{n}-\tx_{n}|\geq \varepsilon,
    \end{align}
then
    \begin{align}\label{eq:converge}
        \sup_{0\leq n\leq N}\frac{|\sigma_{n+1}-\sigma|}{|\sigma_n-\sigma|}\leq\frac{1}2.
    \end{align}
Moreover, $\s_n>0$, for all $n=1,\dots, N+1$.
\end{Thm}

An immediate consequence of \cref{thm:convergence:x} is that if \eqref{cond:non:degen} holds for all $N>0$, then \eqref{eq:converge} enforces exponential convergence to the true parameter value. As previously mentioned, the proof of \cref{thm:convergence:x} is provided in \cref{sect:proofs}. We emphasize, however, that similar statements for \textit{any} of the other combinations of parameters are also available provided that the appropriate datum is also supplied. In particular, one may recover $(\s,\be)$ provided that $(x,z)$ is observed or $(\rho,\be)$, provided that $(y,z)$ is observed. The choice to supply the proof of \cref{thm:convergence:x} is one of efficiency. Of course, all three parameters, $(\s,\rho,\be)$ can also be recovered provided that each of the variables $(x,y,z)$ are known. As this latter situation represents a trivial case, however, one need not apply the proposed algorithm. We point out that a proof can nevertheless be supplied for this trivial case in the same spirit as all the other cases.  Since the statements analogous to \cref{thm:convergence:x} for the other combinations of parameters can be found by making straight-forward adjustments of the analysis that implies the single-parameter case represented by \cref{thm:convergence:x}, we simply provide the statement for one of the other combinations here and refer the reader to \cref{rmk:multi:proof} for additional details. Indeed, our analysis has been organized in such a way to accommodate these other combinations.

\begin{Thm}\label{thm:convergence:xy}
Suppose that $\mu_3=0$ and that $\bt=\be$. Let $M \geq 1$ and $\s_0,\rho_0>0$ such that $|\sigma-\sigma_0| + |\rho-\rho_0| \leq M$, and choose a tolerance $\veps > 0$. There exist $C,C'>0$ such that if $\mu_1,\mu_2$ satisfy
    \begin{align}\label{cond:mu:large:xy}
        \mu_1,\mu_2\geq C\left(1+\s+\rho+\be+R\right)^2,
    \end{align}
and
    \begin{align}\label{cond:mu:converge:xy}
        \mu_1,\mu_2\geq\frac{C'}{\veps^2},
    \end{align}
then if there exists a sequence of $N$ update times $0<t_1<t_2<\dots<t_N$ such that
    \begin{align}\label{cond:non:degen:xy}
      \inf_{0\leq n\leq N+1}\left\{|\ty_{n}-\tx_{n}|, |\tx_n|\right\}\geq \varepsilon,
    \end{align}
then
    \begin{align}\label{eq:converge:xy}
        \sup_{0\leq n\leq N}\frac{|\sigma_{n+1}-\sigma|+|\rho_{n+1}-\rho|}{|\sigma_n-\sigma|+|\rho_n-\rho|}\leq\frac{1}2.
    \end{align}
Moreover, $\s_n,\rho_n>0$, for all $n=1,\dots, N+1$.
\end{Thm}

\subsection{Comparison with direct-replacement data assimilation}\label{sect:direct:replace}

Rather than inserting observations into the system via feedback control as in \eqref{eq:Lor:ng}, one may instead directly substitute the observed quantity into the equations themselves. As a data assimilation algorithm, this was first studied for the Lorenz equations and 2D Navier-Stokes equations in \cite{hayden2011discrete}. We can extend these ideas in an analogous way to the feedback control-based approach developed earlier to parameter recovery.

To fix ideas, assume that $\{x(t)\}_{t\geq 0}$ is given and consider the problem of recovering $\s$ from this information alone. Let $\til{\s}$ represent a guess for $\s$. Then the corresponding modeled system is given by
\begin{empheq}[left=\empheqlbrace]{align}
\label{eq:Lor:dr}
\begin{split}
       \tx &= x\\
       \dot{\ty} &= \rho x - \ty - x\tz  \\
       \dot{\tz} &= x\ty - \bt\tz.
\end{split}
\end{empheq}
Note that the first equation in \eqref{eq:Lor:dr} has no time derivative since we simply replace $\tx$ by $x$ directly.
From \eqref{eq:Lor}, we see that $\s=\dot{x}(y-x)^{-1}$, so we make the approximation
    \begin{align}\notag
        \st=\frac{\dot{x}}{\ty-x}.
    \end{align}
It follows that
    \begin{align}\notag
        |\De\s|<\eps,
    \end{align}
provided that
    \begin{align}\label{cond:non:degen:dr}
        \frac{|\dot{x}||\ty-y|}{|\ty-x||y-x|} < \eps.
    \end{align}
    
Several remarks are in order. Firstly, \eqref{cond:non:degen:dr} may be viewed as an analog to the non-degeneracy condition \eqref{cond:non:degen}. In order for the parameter error $|\De\s|$ to be small using the direct-replacement approach, either $\ty$ must be sufficiently close to $y$ or otherwise, the product of the distance between $\ty$ and $x$ or $y$ and $x$ must be sufficiently large. When \eqref{eq:Lor} exhibits non-trivial dynamics, $|\dot{x}|$ is typically large, potentially making \eqref{cond:non:degen:dr} more difficult to satisfy. From this point of view, the feedback control-based approach provides a degree of flexibility through the tunable parameter $\mu_1$ that is not present in the direct-replacement approach. 

Secondly, the effectiveness of the direct-replacement approach appears to rely on having a sufficient density of observations in time to approximate the time-derivative, $\dot{x}$. Of course, this is not an issue when observations are collected continuously in time. In practical scenarios where measurements are taken at discrete times, however, the time-derivative may be difficult to accurately approximate. The feedback control-based approach, on the other hand, does not require knowledge of time derivatives. Information on the time derivatives of the state variables can be quite powerful. Indeed, it was proved  in \cite{WingardThesis} that knowledge of $\dot{x}$ and $x$ at a single point in time are sufficient to reconstruct $y$ and $z$ at that time.

In spite of these remarks, the direct-replacement approach to parameter recovery warrants further study, both analytically and numerically, especially as a basis of comparison to the feedback control-based approach that is the main focus of this article.

\section{Computational results}\label{sec:comps}

We present a computational study of the parameter learning algorithm described above in Section \ref{sect:multi}. First we conduct a sweep of the Lorenz system's three-dimensional parameter space with observations collected continuously in time. These results illustrate the dependence of the parameter learning algorithm on the dynamical behavior of the reference solution. They additionally demonstrate the robustness of the algorithm to variations of the system parameters. We then proceed to investigate situations beyond those captured in Theorem \ref{thm:convergence:x} and Theorem \ref{thm:convergence:xy}. In particular, we introduce aspects of the problem that are closer to reality, including discrete observations in time, noisy observations, and stochastic forcing in the underlying system. Partial observations with stochastic errors have previously been studied in the setting of the Navier--Stokes equations \cite{bessaih2015continuous, BlomkerLawStuartZygalakis2013}, but under the assumption that the viscosity parameter was already known. The demonstrated effectiveness of the parameter learning algorithm in the presence of stochastic observation error (see Figure \Cref{fig:srb08_noise_part1,fig:srb08_noise_part2}) motivates future work to develop a rigorous justification, perhaps by amending the proof of Theorems \ref{thm:convergence:x} and \ref{thm:convergence:xy}. We refer the reader to \cite{CialencoGlattHoltz2011} for an analytical study of parameter estimation for the 2D stochastically perturbed Navier-Stokes equation, where only finitely many Fourier modes are observed.

Note that data assimilation for the Lorenz equations with sampling of only the $x$ variable was studied both analytically and computationally in \cite{hayden2011discrete} using a replacement scheme rather than a nudging scheme (\cite{hayden2011discrete} also studied model replacement strategies for 2D Navier-Stokes).  In \cite{hayden2011discrete}, it was assumed that the parameters $\sigma$, $\rho$, and $\beta$ were known exactly.  Of course, the goal of the present work is to show that even if the parameters are unknown, the true parameters can be learned from observations using the schemes described above.  Here, we only study parameter learning using a nudging scheme, but parameter learning using replacement schemes will be studied in a future work.

\subsection{Numerical methods}\label{sec:numerical_methods}

The code/data used to run simulations and produce the figures is available at
\begin{center}
    \url{https://github.com/unis-ing/lorenz-parameter-learning}.
\end{center}
For the sake of transparency and the convenience of the reader, a compressed version of the MATLAB script used to run simulations in \Cref{sec:mprwssit} involving sparse-in-time observations, stochastic forcing, and observations contaminated by noise can be found in \Cref{sect:matlab}.

\subsubsection{Setup for continuous sampling (\Cref{sec:splwcs}--\Cref{sec:mplwcs})}
The Lorenz system \eqref{eq:Lor} and the assimilating system \eqref{eq:Lor:ng} were solved using LSODA, which is a variant of the Livermore Solver for Ordinary Differential Equations (LSODE). The LSODA routine switches between nonstiff (Adams type) and stiff (backward differentiation) methods. Upon initialization, LSODA begins with a nonstiff method and then tracks the nature of the underlying ODE to determine whether the stiff or nonstiff solver is more appropriate. The time-step in LSODA is adaptive and chosen to minimize the local error \cite{radhakrishnan1993lsode}. We ensure that the Lorenz system is initialized within the absorbing ball for every simulation reported here by integrating it forward 5 time units from the position $(x,y,z) = (60,60,10)$. The assimilating system is initialized at $(\tx,\ty,\tz) = (0.1,0.1,0.1)$.

\subsubsection{Setup for sparse sampling (\Cref{sec:mprwssit} and \Cref{sec:mprwssit_noise})}
The Lorenz system \eqref{eq:Lor} and the assimilating system \eqref{eq:Lor:ng} were simulated using the explicit first order Euler method. For these experiments, we avoided using higher-order Runge-Kutta methods because such multi-stage methods, when applied to nudging-based schemes, violate the principle of having so-called ``identical discrete dynamics'' and can misrepresent the error (see, e.g., the discussion in Section 4 of \cite{olson2008determining}). In addition to the forward Euler method, we also treated the linear terms implicitly using an exponential time-differencing (ETD) method, but saw no significant differences in the results.  Hence, for simplicity, only the explicit forward Euler time-stepping simulations are reported here.

\subsection{Single-parameter learning with continuous sampling}\label{sec:splwcs}
As the first point of comparison, we suppose that only $\sigma$ is unknown (both $\beta$ and $\rho$ are known exactly) and that $x(t)$ is continuously observed, i.e., $x(t)$ is known for all values of $t$. This is precisely the situation hypothesized in \cref{thm:convergence:x} above. Recall from \cref{sect:derive} that $\frac{du}{dt}$ was implicitly assumed to be negligible in deriving the parameter update formulas. This suggests that one must ensure this property holds prior to making an update to the parameter. In principle, it is possible to enforce this property in the numerical setting by approximating $\frac{du}{dt}$ via finite differences, although we have observed parameter convergence from simply enforcing $u \approx 0$. To enforce the latter condition, we implement thresholds on $u$ of the form $|u(t)| \leq \theta$, then systematically decrement $\theta$ with each parameter update. An example of a decrementing mechanism would be to fix a constant $0<d<1$ and update the threshold as $\theta \mapsto d\theta$ after each parameter update. However, ``thresholding" in this way requires detailed tuning of $d$ and the initial value of $\theta$. We propose an alternative method for threshold selection that leverages the observational data and does not introduce additional algorithmic parameters.

Let $t_n$ be the time of the previous update and $t>t_n$ be the current time. Let $b_n,~m_n \in \mathbb{R}$ be a linear fit of the dataset $\{\log{|u(s)|}\,:\,t_n\leq s< t\}$. That is, $b_n,~m_n$ minimize the mean squared error over $\{\log{|u(s)|} - (m_n s + b_n)\,:\,t_n\leq s< t\}$. The threshold $\theta$ at time $t$ is then defined to be
\begin{align}\label{eq:threshold}
    \theta(u,t_n,t) = \exp{m_n t + b_n}.
\end{align}
In requiring $|u(t)|$ to be bounded by $\theta$ as defined, one eliminates candidates for the next update time $t_{n+1}$ in which $|u(t_{n+1})|$ is a local maximum. Since the formula was derived implicitly assuming that both $u$ and $\frac{du}{dt}$ are small, the best candidate times for updates are expected to be in neighborhoods of local minima of $u$. Further comments regarding the threshold $\theta$ can be found in \cite{ng2021dynamic}.

The entire procedure for estimating $\sigma$ with continuous observations in $x(t)$ is given in Algorithm \ref{alg:cts_sampling_sigma_recovery}. The algorithm is robust to a wide range of initial estimates, provided that $\mu_1$ and $T_R$ are taken large enough (see \cite{ng2021dynamic}). We were able to recover $\sigma$ from the reference values $\sigma=10,\,\rho=28,\,\beta=8/3$ up to an error of magnitude $\mathcal{O}\left(10^{-13}\right)$ with an initial parameter error of magnitude $\mathcal{O}\left(10^3\right)$. In Figure \ref{fig:cts_sampling_demo}, we apply Algorithm \ref{alg:cts_sampling_sigma_recovery} and observe the evolution of the ``position error", i.e., $\norm{(u,v,w)}$, and ``velocity error", i.e., $\norm{(\gam,\de,\eta)}$. The analytically derived upper bounds on both position and velocity error (Corollary \ref{cor:diff:energy} and Corollary \ref{cor:diff:dot:energy}) hold remarkably well. On the other hand, the restriction \eqref{cond:mu:large} specified for $\mu_1$ was calculated to be on the order of $\mathcal{O}(10^7)$, which we found to be quite excessive; in practice, a sufficient value for $\mu_1$ to accurately infer $\s$ is $\mathcal{O}(10^2)$. It should be noted that $\mu_1$ is an algorithmic parameter, not a physical one, so it may be tuned as the user prefers, so long as solutions remain stable.  For example, for rapid convergence, it is often desirable to choose it as large as linear stability will allow, that is, $\mu_1\lesssim2/\Delta t$, where $\Delta t$ is the (largest) time-step.

\begin{figure}[ht]
\centering
\includegraphics[width=.68\textwidth]{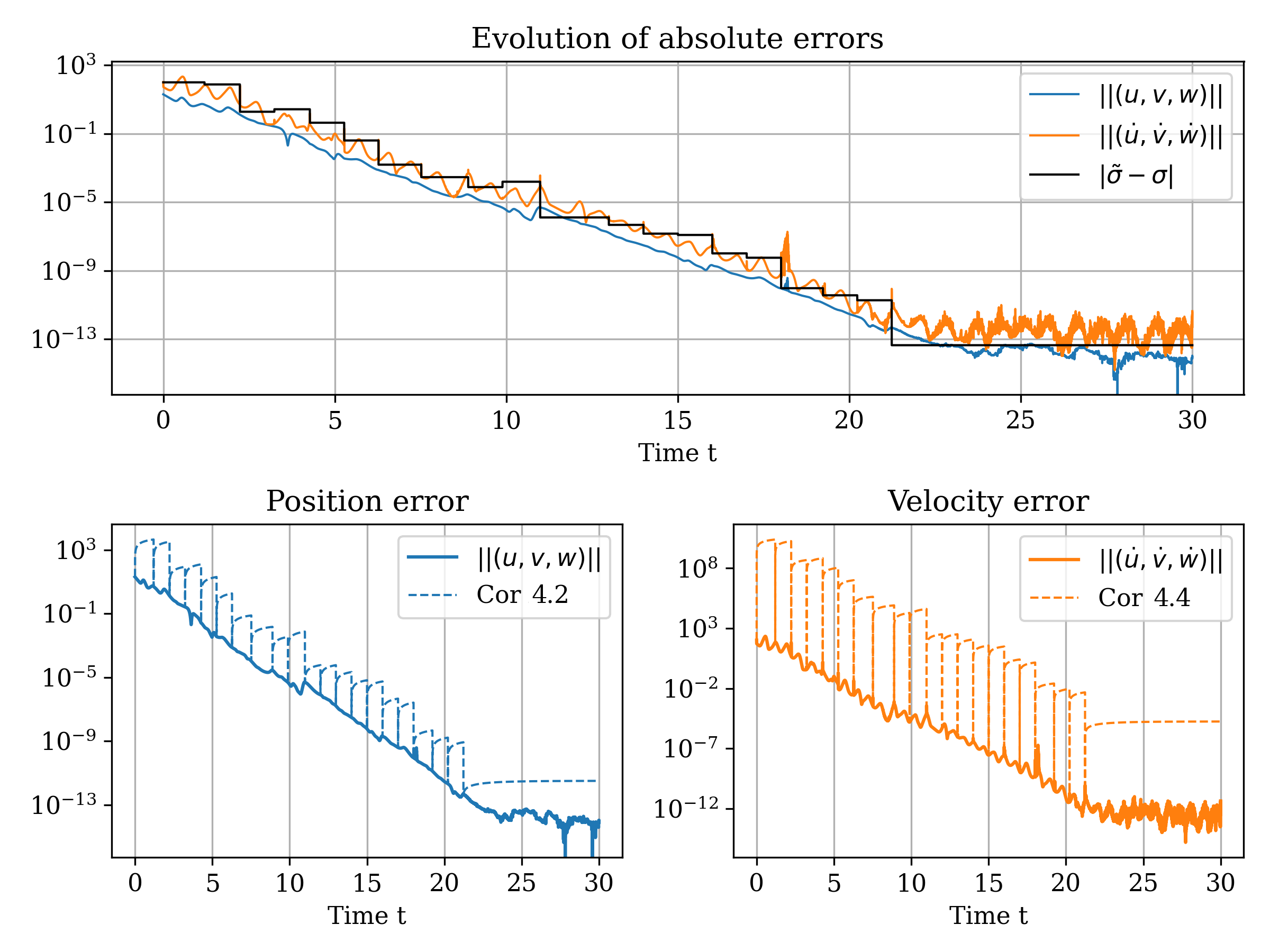}
\caption{The parameter learning algorithm is applied to the true parameters $\sigma=10,\, \rho=28,\, \beta=8/3$ with $\rho,\beta$ known and $\sigma$ recovered from continuous observations in $x(t)$. The initial guess is $\sigma_0 = \sigma+100$ and the algorithm parameters were set to $\mu_1=500$, $T_R=1$. The analytically derived upper bounds on position and velocity error from Corollary \ref{cor:diff:energy} and Corollary \ref{cor:diff:dot:energy} are shown to hold remarkably well.}
\label{fig:cts_sampling_demo}
\end{figure}

\begin{figure}[ht]
    \centering
    \includegraphics[width=.43\textwidth]{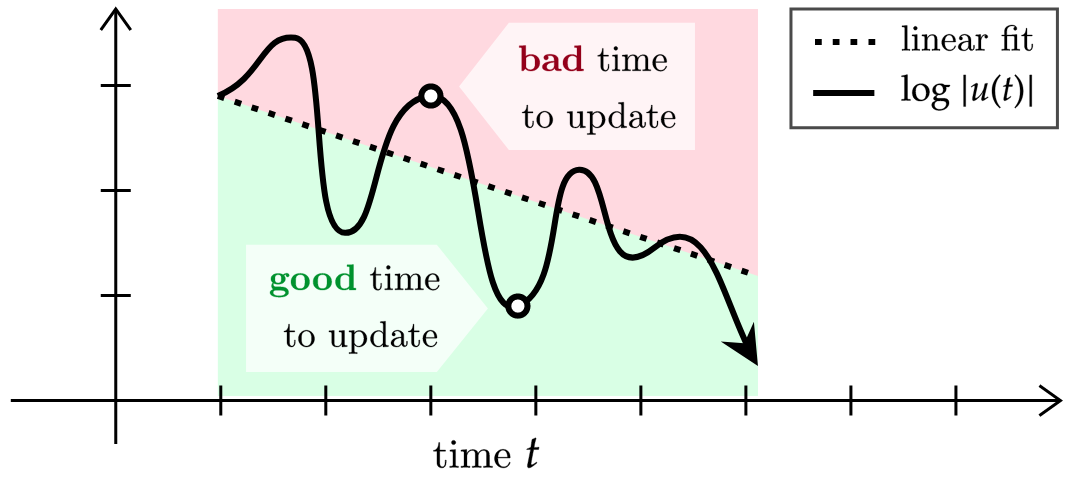}
    \caption{Schematic of the threshold defined by \eqref{eq:threshold}.}
    \label{fig:log_fit_schematic}
\end{figure}

\begin{figure}[htbp]
\normalem 
\begin{center}
\begin{minipage}{0.75\linewidth}
    \begin{algorithm}[H]
    \SetKwInOut{Inputs}{Inputs}
    \SetKwFunction{Polyfit}{np.polyfit}
    \SetAlgoLined
    \Inputs{Initial estimate $\sigma_0$; relaxation period $T_R \geq0$; relaxation parameter $\mu_1>0$}
    $t \leftarrow 0$\;
    $t_n \leftarrow 0$ \tcp*{Time of last update}
    $\widetilde{\sigma} \leftarrow \sigma_0$\;
    \For{$i=0,1,\ldots$}{
        \If{$|\xt(t)-x(t)|>0$}{
            \If{$\xt(t)\neq\yt(t)$ {\bf and} $t-t_n\geq T_R$ {\bf and} $|\widetilde{x}(t)-x(t)| \leq \theta(u,t_n,t)$}{
                    $\widetilde{\sigma} \leftarrow \st - \mu_1 u(t) \,/\, (\yt(t)-\xt(t))$\; 
                    $t_n \leftarrow t$\;
            }
        }
        $t \leftarrow t+\Delta t$\;
        Integrate \eqref{eq:Lor}, \eqref{eq:Lor:ng} forward to $t$\;
    }
    \caption{Recovering $\sigma$ from continuous observations $\{x(t)\}_{t\geq0}$ \label{alg:cts_sampling_sigma_recovery}}
    \end{algorithm}
\end{minipage}
\end{center}
\ULforem 
\end{figure}

\FloatBarrier

\subsection{Minimum \texorpdfstring{$\mu$}{\textmu} for parameter learning}

The analysis in \cref{sect:proofs} suggests that $\mu_1$ can be decreased as the nudged system synchronizes with the true system and the estimate $\st$ improves over time. In our simulations we have chosen $\mu_1$ to be constant in time for the sake of simplicity. Nonetheless, we found that the numerical lower bound for a constant $\mu_1$ is much lower than the restriction stated in Theorem~\ref{thm:convergence:x}. We compute the numerical lower bound in Table \ref{tab:minimal_mu}, denoted $M_c$, by taking the largest multiple of 10 such that the parameter converges to the true value within $\mathcal{O}\left(10^{-5}\right)$ error. We refer the interested reader to \cite{ng2021dynamic} for the behavior of $M_c$ over a larger region of the $(\sigma,\rho)$-plane.

\begin{table}[ht]
    \[
    \begin{array}{||c c c c||} 
        \multicolumn{4}{c}{\sigma=10} \\ \hline
        \rho & \mu_c & M_c & |\Delta\sigma| \\ [0.5ex] \hline\hline
        30 & 2.6525\times10^{7} & 20 & 6.2368\times10^{-12} \\ \hline
        50 & 1.3384\times10^{8} & 20 & 1.2981\times10^{-9} \\ \hline
        100 & 1.5092\times10^{9} & 50 & 2.0904\times10^{-11} \\ \hline
        150 & 6.7527\times10^{9} & 60 & 1.3397\times10^{-9} \\ \hline
        200 & 2.0036\times10^{10} & 60 & 2.9829\times10^{-11} \\ \hline
    \end{array}
    \quad
    \begin{array}{||c c c c||} 
        \multicolumn{4}{c}{\sigma=50} \\ \hline
        \rho & \mu_c & M_c & |\Delta\sigma| \\ [0.5ex] \hline\hline
        30 & - & - & - \\ \hline
        50 & 1.0310\times10^{9} & 210 & 1.9910\times10^{-9} \\ \hline
        100 & 5.2165\times10^{9} & 130 & 8.4269\times10^{-9} \\ \hline
        150 & 1.6484\times10^{10} & 60 & 7.3339\times10^{-6} \\ \hline
        200 & 4.0240\times10^{10} & 90 & 1.7521\times10^{-8} \\ \hline
    \end{array}
    \]
    \caption{The smallest $\mu_1$ resulting in $|\Delta\sigma| \leq 10^{-5}$ is estimated to the largest multiple of 10 (denoted as $M_c$). The analytically derived lower bound for $\mu_1$ in Theorem~\ref{thm:convergence:x} is computed by evaluating (\ref{cond:mu:diff}) at $t=0$ (denoted as $\mu_c$). The values are omitted if $\sigma$ could not be recovered from any of the tested values of $\mu_1$ (see Section \ref{sect:dyn:dep} for further discussion). The initial estimate was set to $\sigma_0=\sigma+10$, each simulation was run up to $t=200$ time units, and the relaxation period parameter was fixed at $T_R=5$.}
    \label{tab:minimal_mu}
\end{table}

\subsection{Dependence on dynamical behavior}\label{sect:dyn:dep}

\begin{figure}[ht]
    \centering
    \includegraphics[height=6cm]{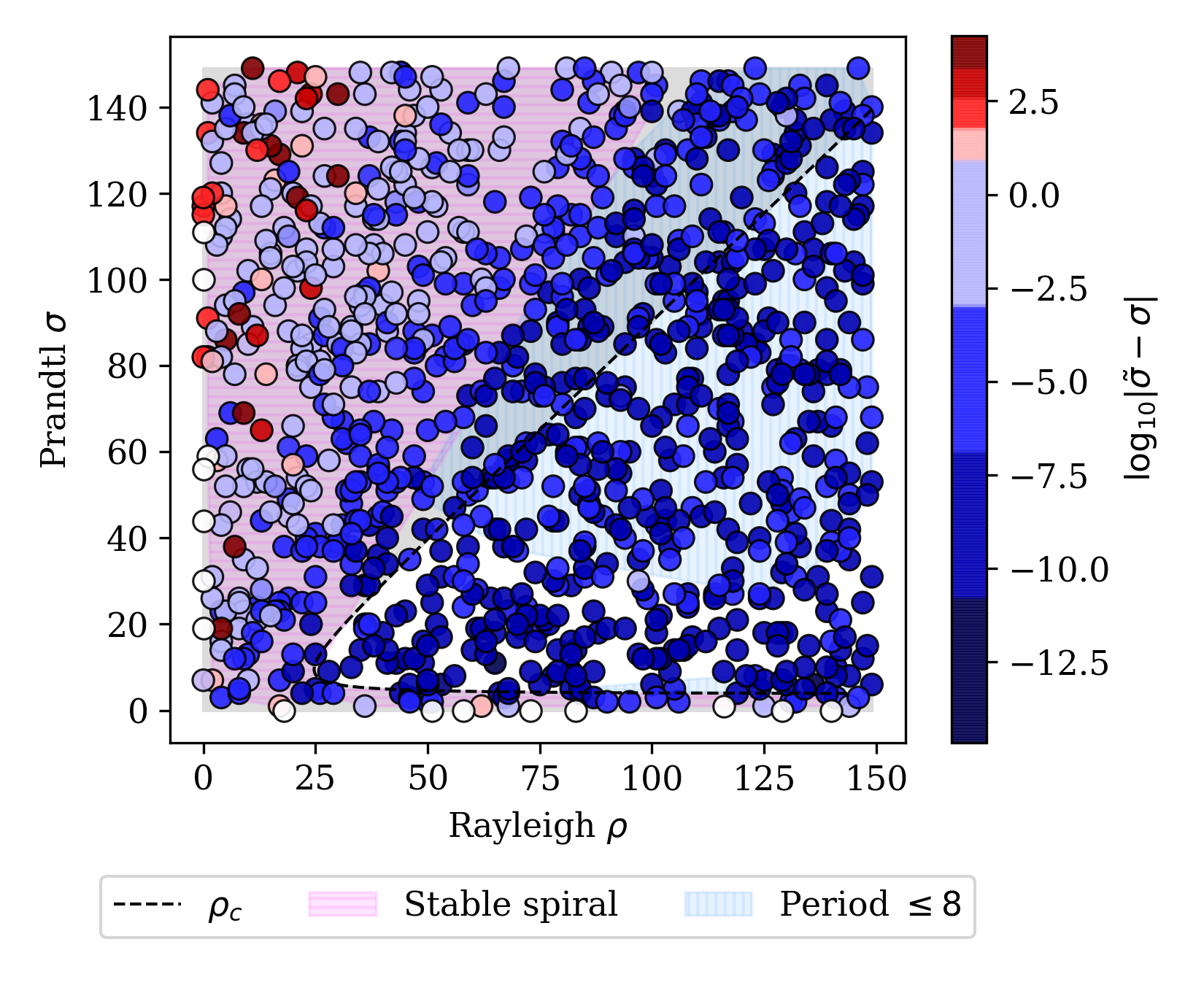}
    \includegraphics[height=6cm]{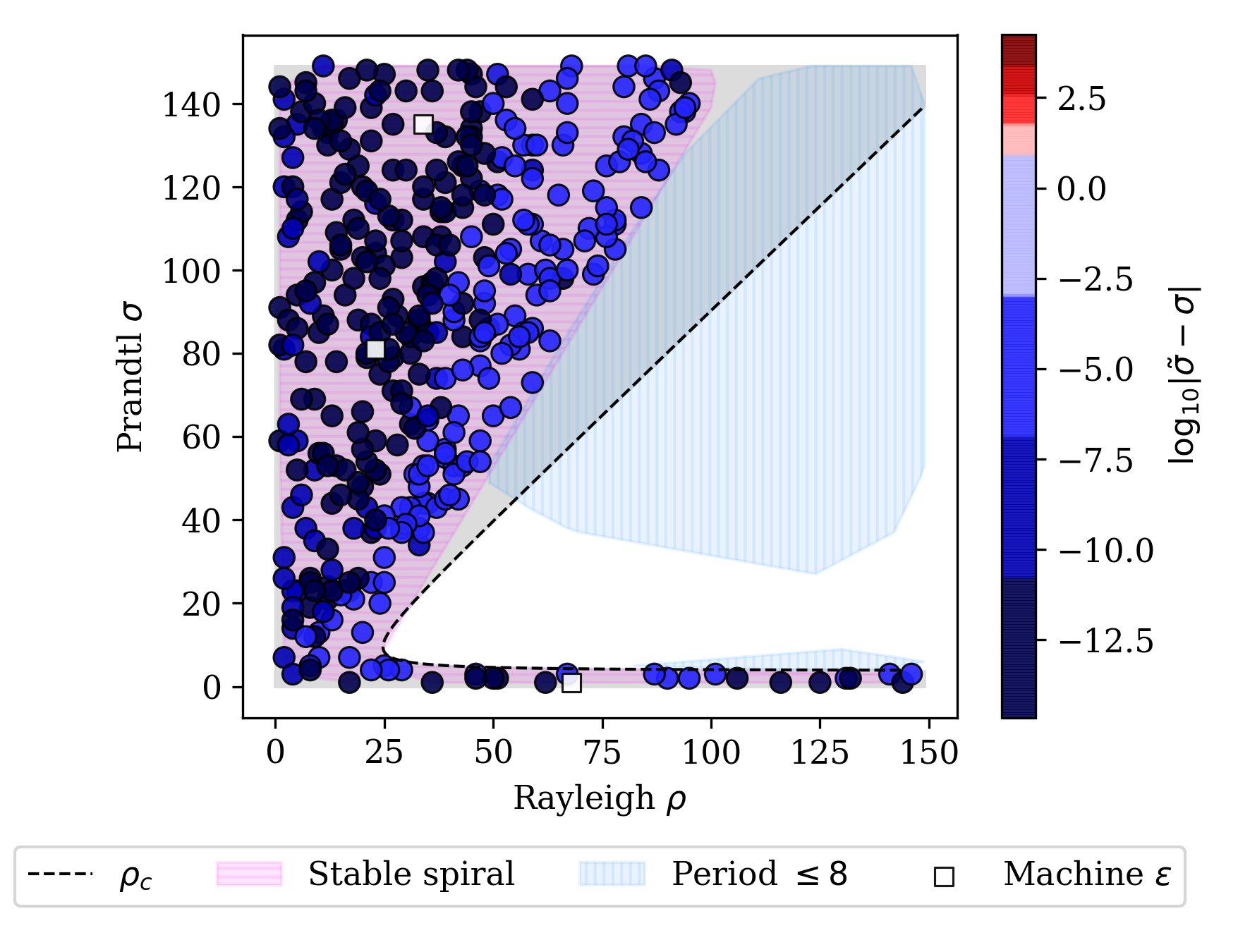}
    \caption{(Left) The parameter learning algorithm is used to recover $\sigma$ from 1,000 randomly sampled pairs $(\rho,\sigma) \in [0,150]^2$, with $\beta=8/3$ fixed. The initial estimate used is $\sigma_0=\sigma+10$, and the algorithm parameters are fixed at $\mu_1=10,000$ and $T_R=5$. Each simulation is run to $t=75$ time units. The color corresponds to the resulting absolute parameter error $|\st-\sigma|$: red signifies $|\st-\sigma|>|\sigma_0-\sigma|$; white signifies $|\st-\sigma|=|\sigma_0-\sigma|$; blue signifies $|\st-\sigma|<|\sigma_0-\sigma|$. The period of each solution is computed using the Poincar\'e surface of section method described in \cite{dullin2007extended}. (Right) For each $(\rho,\sigma)$ where $P_\pm$ is stable, we observe the variable $z_\tau = z-\sigma-\rho$ instead of $x$ and use the alternate formula \eqref{eq:sig_recovery_transl} to recover $\sigma$.}
    \label{fig:singleparam}
\end{figure}

As we might anticipate for any type of inverse problem with a complicated forward map, the $\sigma$-learning problem does not always have a unique solution. For example, if $x \approx y$ for all $t\geq 0$, then there is no unique value of $\sigma$ which satisfies $\frac{dx}{dt}=\sigma(y-x)\approx 0$. This is exactly the scenario that occurs when one of two nontrivial fixed points of the Lorenz equations is stable, since the nontrivial fixed points
\begin{align}
    P_\pm &= (x_\pm, ~y_\pm, ~z_\pm) = (\pm\sqrt{\beta(\rho-1)}, ~\pm\sqrt{\beta(\rho-1)}, ~\rho-1).\notag
\end{align}
satisfy $x_\pm=y_\pm$. Note that the stability of the Lorenz equations can be analyzed using the usual linearization technique (see \cite{lorenz1963deterministic}). The origin is a fixed point for all parameters and is stable for $0\leq\rho> 1$. The fixed points at $P_\pm$ exist for $\rho>1$ and are linearly stable up to a critical value $\rho=\rho_c$, where stability is lost in a subcritical Hopf bifurcation. This critical value is given by $\rho_c = \sigma (\sigma+\beta+3) / (\sigma-\beta-1)$. 

Thus, if $x(t)- y(t)\not\approx0$ at a single time $t$, the issue of non-uniqueness is not expected to play a role. Indeed, the purpose of the non-degeneracy condition  \eqref{cond:non:degen} in \cref{thm:convergence:x} is to exclude such a pathological scenario. Note, however, that this non-degeneracy is only enforced on the nudged system, and \textit{not} on the observations; we believe that this allows an additional layer of flexibility in our algorithm.

A parameter sweep of the $(\rho,\sigma)$-plane with the third parameter fixed at $\beta=8/3$ confirms that the parameter learning algorithm implementing \eqref{def:update} as the parameter update formula performs noticeably worse or altogether fails precisely when a stable fixed point exists (see Figure \ref{fig:singleparam}). The same experiment demonstrates that the algorithm is robust to variations in the $(\rho,\sigma)$-plane, as long as $\mu_1$ is sufficiently large and the reference solution exhibits chaotic or at least presents non-degenerate dynamics. Comprehensive computational studies such as \cite{barrio2007study} and \cite{dullin2007extended} show that a nontrivial global attractor exists for an unbounded subset of the full three-dimensional Lorenz parameter space, and the current investigation indicates that the parameter learning algorithm developed here should work for all of these parameters.

All this suggests that the non-uniqueness of the inverse problem plays a pronounced role depending on which parameter is being inferred and which variable(s) of the Lorenz system is (are) being observed. The particular case described above, where $\s$ is being inferred and the $x$-variable is being observed, is an explicit scenario in which the non-uniqueness appears to play a significant role. If one instead observes the $z$-variable, an \textit{alternative} update formula for $\s$ may be used in place of the one proposed in \eqref{def:update}, which eliminates the non-uniqueness. Indeed, consider the `translated' state variable $z_\tau = z-\sigma-\rho$. If one were to observe $\{z_\tau(t)\}_{t\geq 0}$ rather than $\{x(t)\}_{t\geq0}$, then $\sigma$ could be learned using the following equation for the translated fixed points:
\begin{align}
    P_{\pm,\tau} = (x_{\pm,\tau}, ~y_{\pm,\tau}, ~z_{\pm,\tau}) = (\pm\sqrt{\beta(\rho-1)}, ~\pm\sqrt{\beta(\rho-1)}, ~-\sigma-1).\notag
\end{align}
When $P_{\pm,\tau}$ are stable, the $P_{\pm,\tau}$ equation yields the alternate recovery formula
\begin{align}\label{eq:sig_recovery_transl}
    \sigma = -z_{\pm,\tau}-1 \approx -z_\tau(t)-1.
\end{align}
In this situation, the corresponding nudged equation is given by \eqref{eq:Lor:ng} with $\rt=\rho$, $\bt=\be$, $\mu_1=\mu_2=0$, and $\mu_3>0$, that is, nudging is only implemented in third variable. We see that for each pair $(\rho,\sigma)$ where $P_\pm$ is stable, applying the modified update formula results in the successful recovery of $\sigma$ (see Figure \ref{fig:singleparam}).

\subsection{Multi-parameter learning with continuous sampling}\label{sec:mplwcs}

So far we have only seen that the algorithm described in \cref{sect:derive} with the update formula \eqref{def:update} is effective for estimating a single parameter, namely $\s$. It can just as easily be adapted to recover two parameters simultaneously provided that the appropriate state variables are observed. In particular, each of the system parameters $\sigma$, $\rho$, and $\beta$, in any combination, can be estimated from continuous observations in $x(t)$, $y(t)$, and $z(t)$ respectively. We modify the update condition by defining $\theta_1, \theta_2, \theta_3$ using the log-linear fit procedure in the $x$, $y$, and $z$ coordinates. Depending on which two parameters are unknown, the two corresponding conditions among
\begin{align}
    |u(t)| \leq \theta_1(u,t_n,t), \quad |v(t)| \leq \theta_2(v,t_n,t), \quad |w(t)| \leq \theta_3(w,t_n,t),\notag
\end{align}
are enforced. As before, we perform a sweep of the parameter plane where the remaining third parameter is fixed. We observe that recovery distinctly fails in both parameters when a stable fixed point occurs, in accordance with what was observed in \cref{sect:dyn:dep} (see Figure \ref{fig:two_parameter_recovery}).

\begin{figure}[ht]
    \centering
    \includegraphics[height=4.9cm]{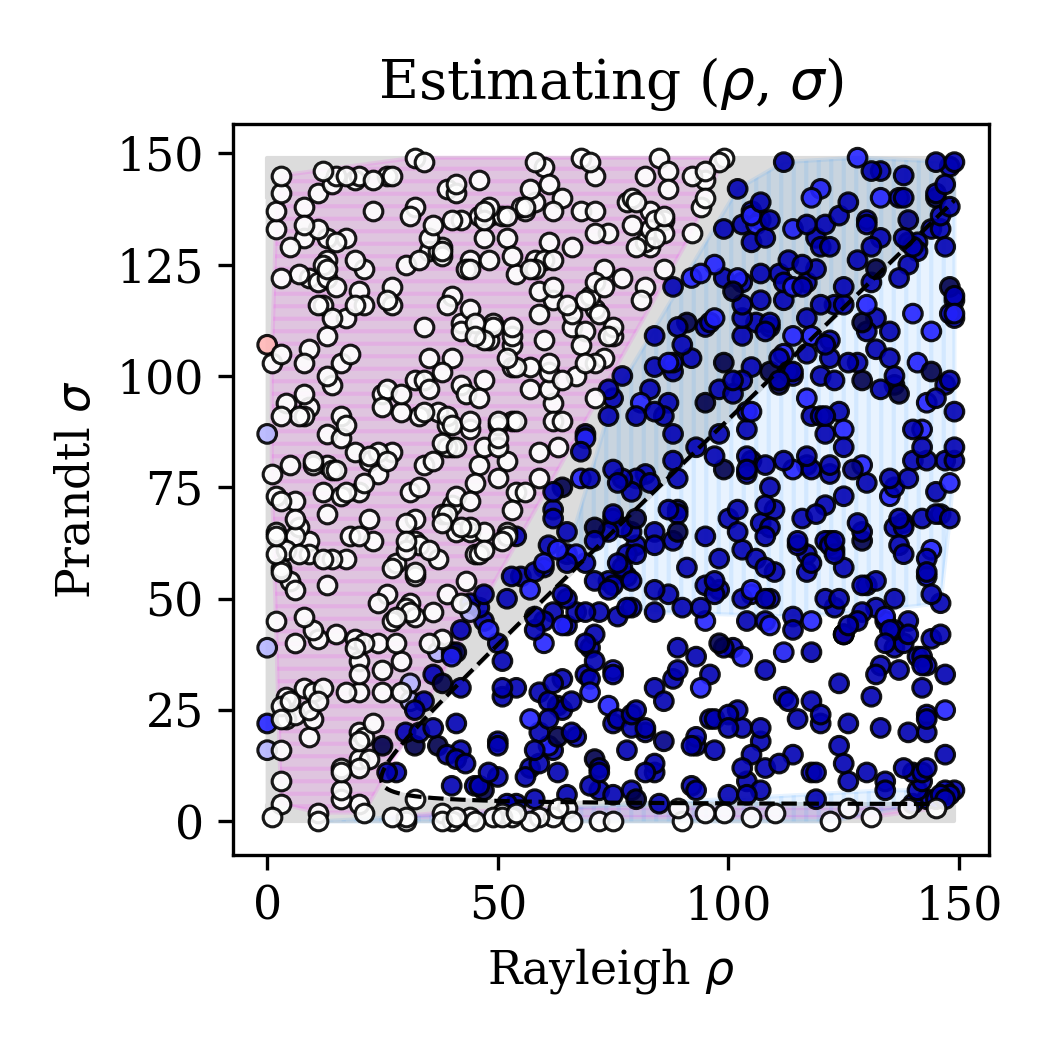}
    \includegraphics[height=4.9cm]{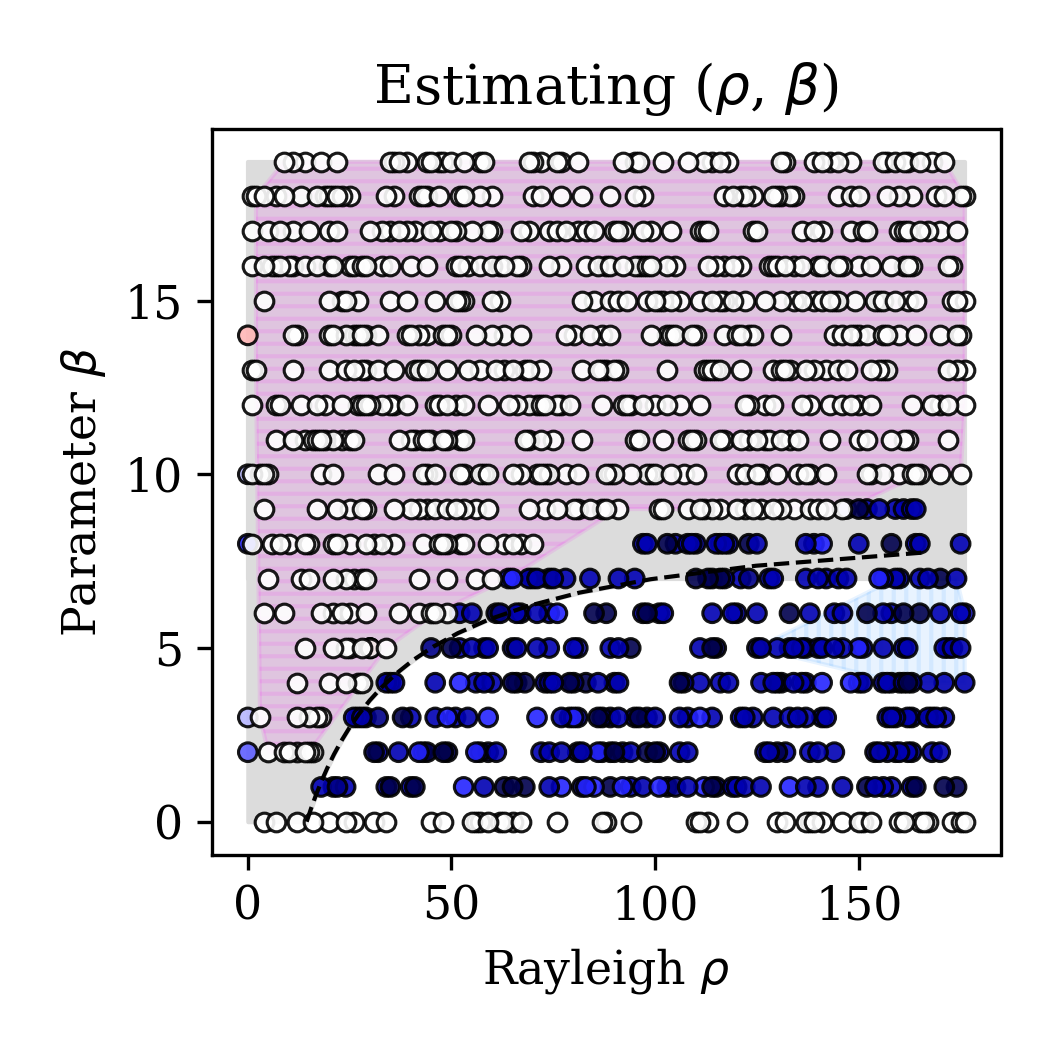}
    \includegraphics[height=4.9cm]{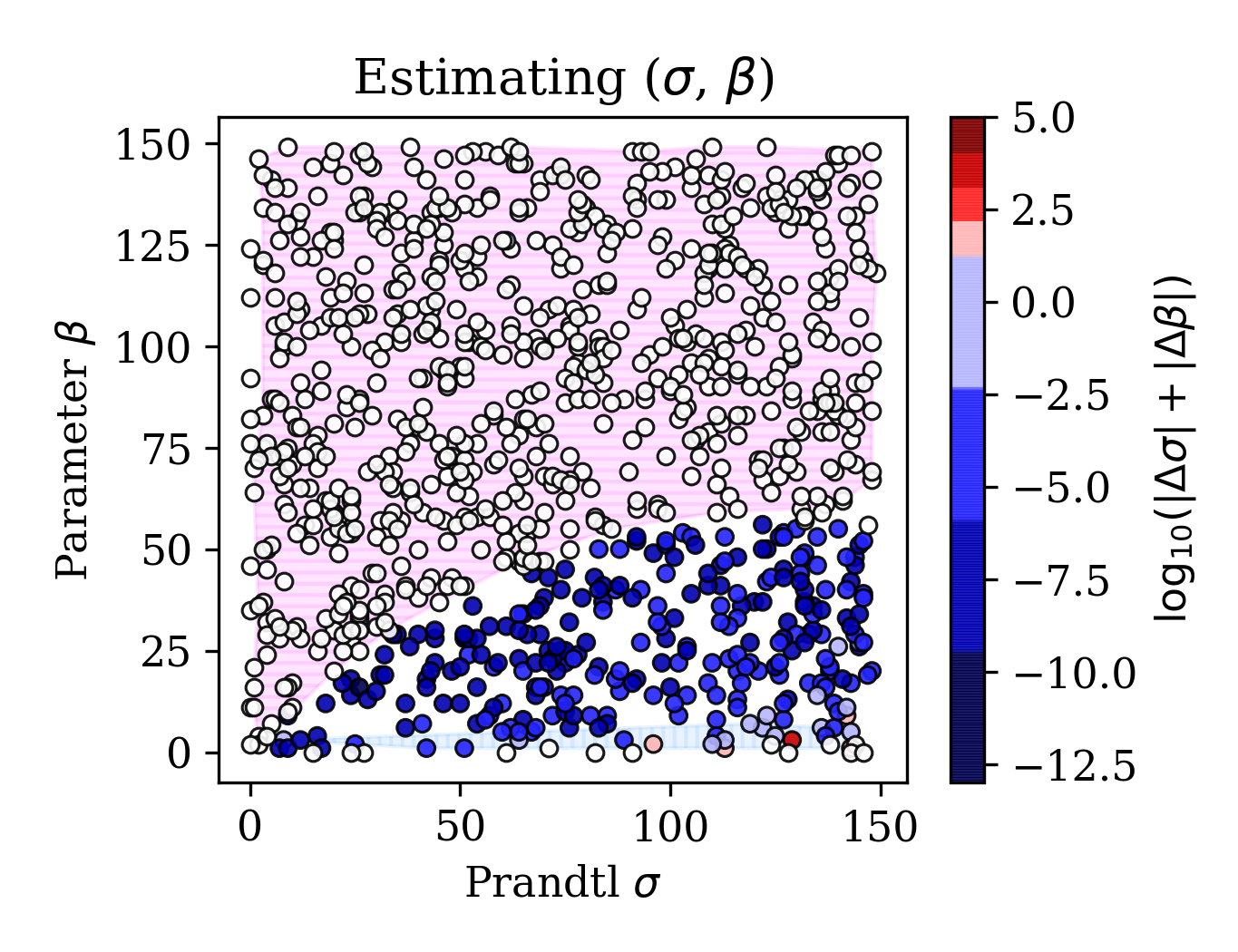}
    \caption{Two parameters are recovered simultaneously. When applicable, the third parameter is fixed at $\sigma=10,~\rho=300,~\beta=8/3$. The initial parameter guess is 10 above the true value. The two relaxation parameters, corresponding to the unknown parameters, are fixed at $10,000$, and the relaxation period is $T_R=5$. Each simulation is run out to $t=50$ time units.}
    \label{fig:two_parameter_recovery}
\end{figure}

\FloatBarrier

\subsection{Multi-parameter recovery with sparse sampling}\label{sec:mprwssit}
In this section, we demonstrate that the parameter learning algorithms discussed above also work with sparse-in-time observations. That is, we observe the state of the system \eqref{eq:Lor} at discrete times that are widely spaced in comparison to other time scales in the system, as well as the numerical time-step.  This modification is non-trivial and leads to some surprising differences from the case of continuous-in-time observations, as discussed below.

Several issues arise at the implementation level, perhaps unexpectedly, from combining intermittent observations with dynamical parameter correction via \eqref{def:update}.  In particular, a balance must be struck between three distinct time scales: the time between observations $\Delta t_{\text{obs}}$, the time between parameter updates $\Delta t_{\text{param}}$, and the length of time that the nudging algorithm is applied $\Delta t_{\text{AOT}}$ (this does not include the time-step for the ODE, for which we used $\Delta t=\texttt{0.0001}$).  For instance, in our simulations, we observe every $500~\Delta t$, but update the parameter(s) every $20,000~\Delta t$ (i.e. $\Delta t_{\text{obs}}=\texttt{0.05}$ and $\Delta t_{\text{param}} = \texttt{2.0}$).  In particular, we found that the values of $\mu_i$ ($i=1,2,3$) used in the parameter update formulas \eqref{def:update} needed to be smaller than those used in \eqref{eq:Lor:ng} by several orders of magnitude. Otherwise, the simulation was not stable.  Denoting $\mu_i^{{\text{AOT}}}$ for the $\mu_i$ used in \eqref{eq:Lor:ng}, and $\mu_i^{\text{param}}$ for those used in \eqref{def:update}, we used $\mu_i^{\text{param}} = \mu_i^{{\text{AOT}}}\Delta t$.  This is an unexpected observation\footnote{In the case of continuous observations (or, more accurately, $\Delta t_{\text{obs}}=\Delta t$), numerical instability did not arise when $\mu_i^{\text{param}} = \mu_i^{{\text{AOT}}}$. Hence, the condition $\mu_i^{\text{param}} < \mu_i^{{\text{AOT}}}$ is only needed in the case of intermittent observations.} that has important implications for real-world implementations; hence, it will be the subject of a future larger-scale investigation.

The primary results of our investigations are shown in Figure \ref{fig:srb08}.  We see that in all cases except when only $z$ is observed with the unknown parameter $\beta$, the true solution and unknown parameters were recovered exponentially fast to near machine precision. We note that the results displayed here hold qualitatively for a variety of different choices of the true parameters $\sigma,~\rho$ and $\beta$, including large values of $\rho$ which lead to much more chaotic dynamics. For the sake of brevity, we only present results with the standard Lorenz parameters $\sigma=10,\,\rho=28,\,\beta=8/3$, and the initial guesses of the parameters given by $\til{\sigma}=0.8\sigma,\,\til{\rho}=0.8\rho,\,\til{\beta}=0.8\beta$.  However, we emphasize that the results reported here are relatively independent of the initial guess for each of the parameters.

\begin{figure}[ht]
\centering
\begin{subfigure}{.32\textwidth}
  \centering
  \includegraphics[width=1.0\linewidth,trim = 0mm 0mm 0mm 0mm, clip]{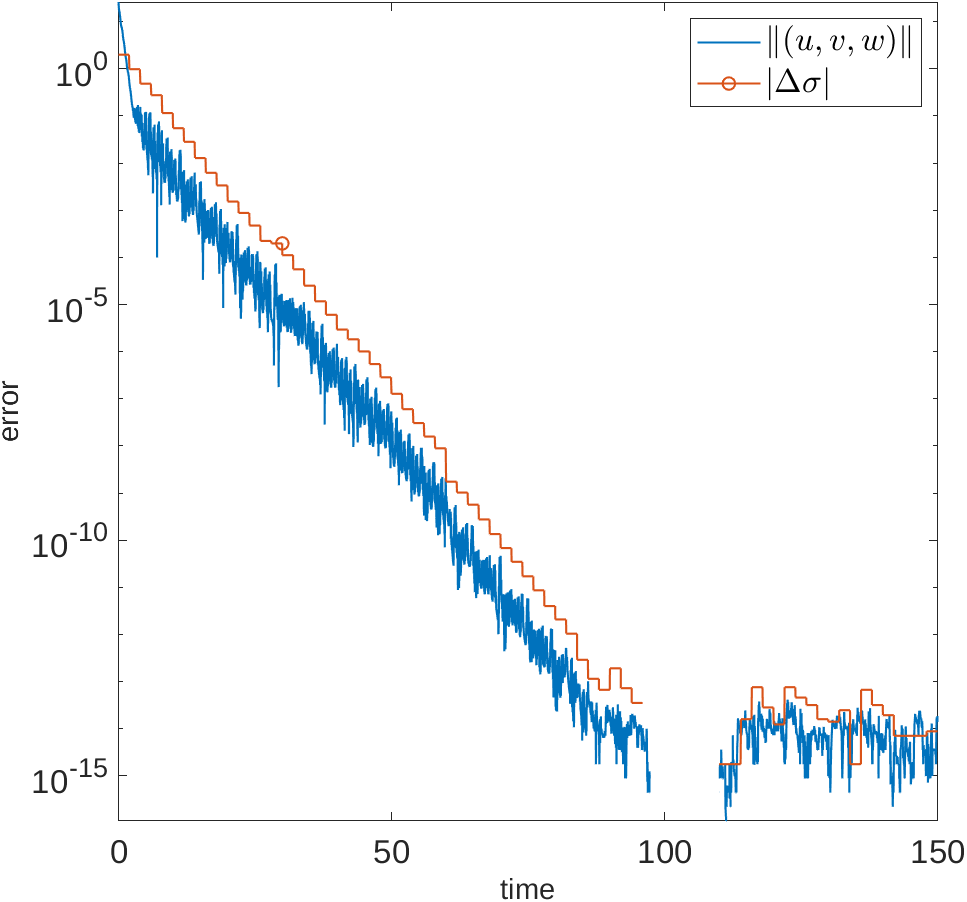}
  \caption{\label{subfig:srb08_100}$\widetilde{\sigma}=0.8\sigma$, $\widetilde{\rho}=\rho$,  $\widetilde{\beta}=\beta$}
\end{subfigure}
\begin{subfigure}{.32\textwidth}
  \centering
  \includegraphics[width=1.0\linewidth,trim = 0mm 0mm 0mm 0mm, clip]{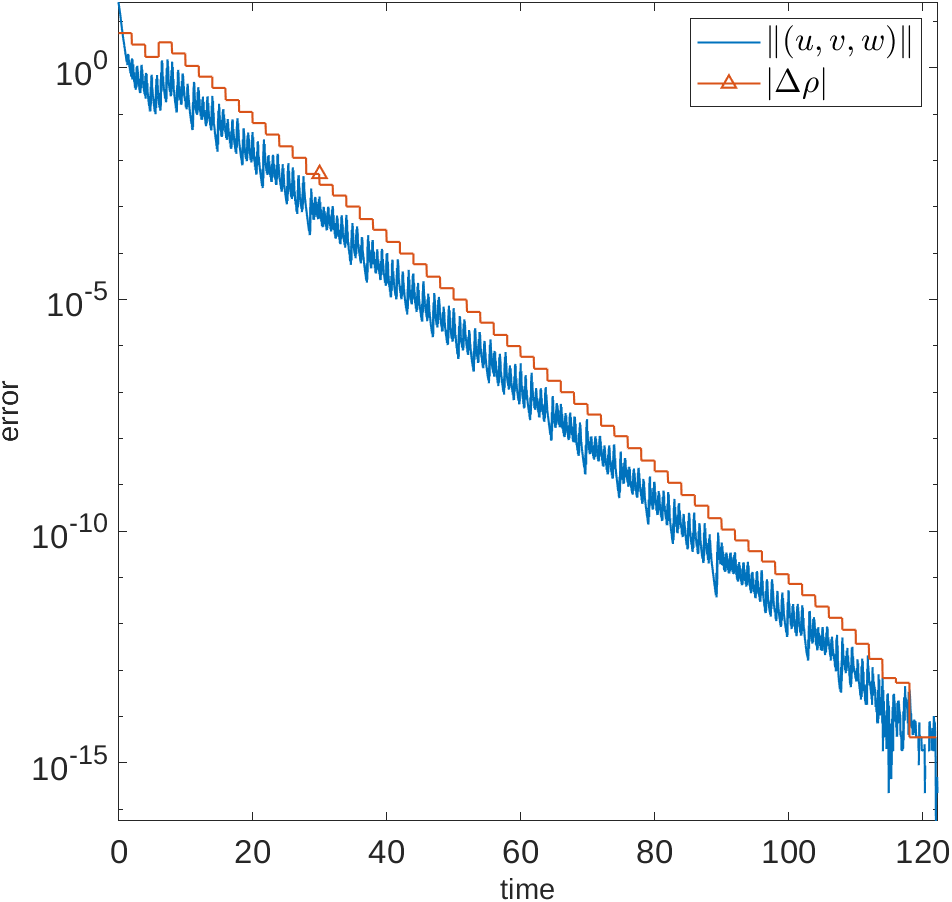}
  \caption{\label{subfig:srb08_010}$\widetilde{\sigma}=\sigma$, $\widetilde{\rho}=0.8\rho$,  $\widetilde{\beta}=\beta$}
\end{subfigure}
\begin{subfigure}{.32\textwidth}
  \centering
  \includegraphics[width=1.0\linewidth,trim = 0mm 0mm 0mm 0mm, clip]{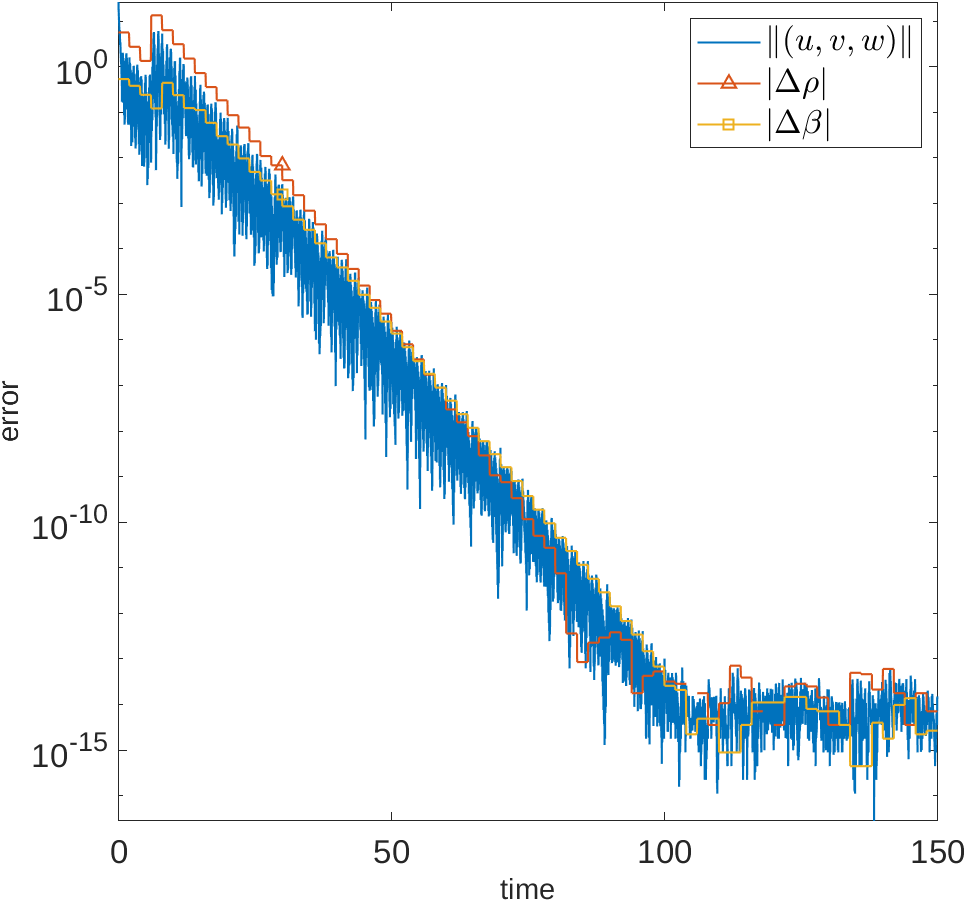}
  \caption{\label{subfig:srb08_011}$\widetilde{\sigma}=\sigma$, $\widetilde{\rho}=0.8\rho$,  $\widetilde{\beta}=0.8\beta$}
\end{subfigure}
\begin{subfigure}{.32\textwidth}
  \centering
  \includegraphics[width=1.0\linewidth,trim = 0mm 0mm 0mm 0mm, clip]{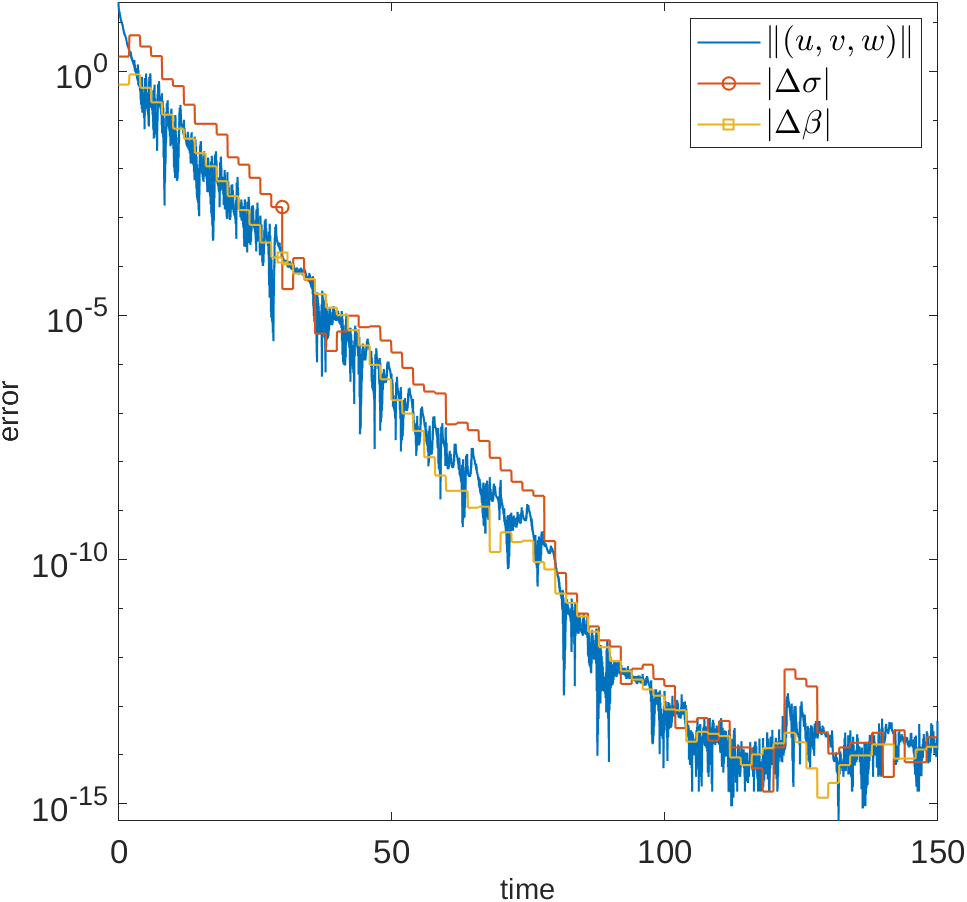}
  \caption{\label{subfig:srb08_101}$\widetilde{\sigma}=0.8\sigma$, $\widetilde{\rho}=\rho$,  $\widetilde{\beta}=0.8\beta$}
\end{subfigure}
\begin{subfigure}{.32\textwidth}
  \centering
  \includegraphics[width=1.0\linewidth,trim = 0mm 0mm 0mm 0mm, clip]{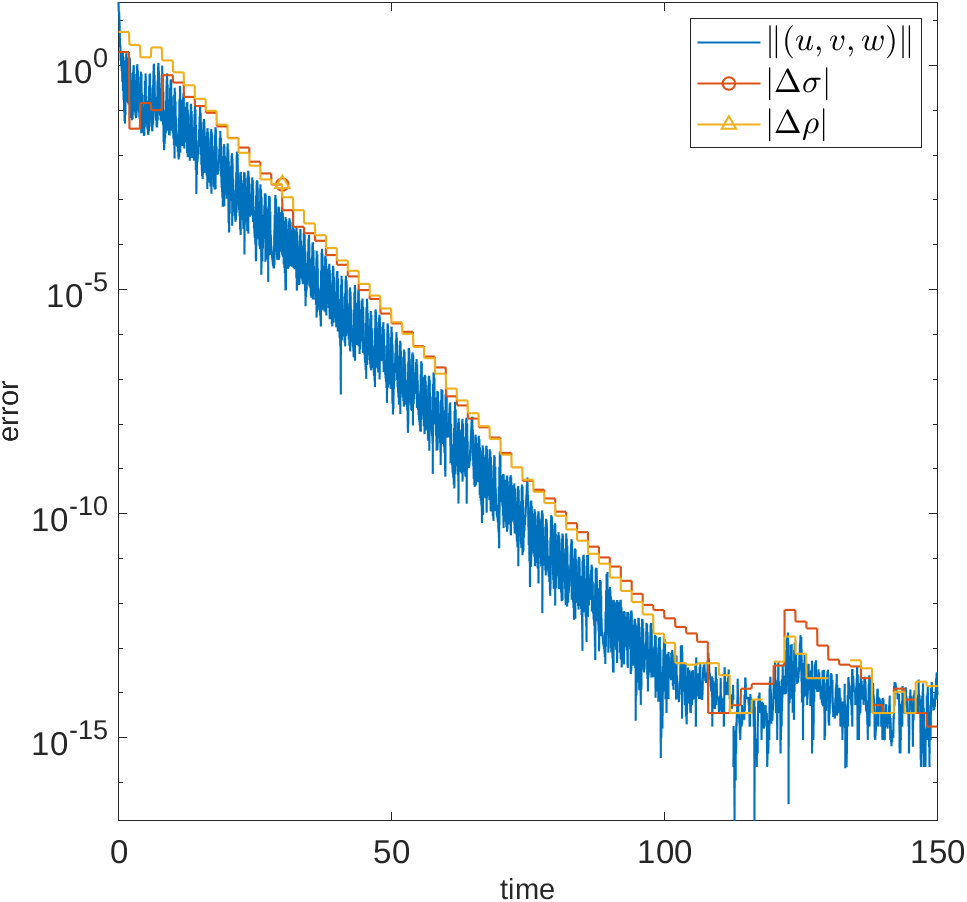}
  \caption{\label{subfig:srb08_110}$\widetilde{\sigma}=0.8\sigma$, $\widetilde{\rho}=0.8\rho$,  $\widetilde{\beta}=\beta$}
\end{subfigure}
\begin{subfigure}{.32\textwidth}
  \centering
  \includegraphics[width=1.0\linewidth,trim = 0mm 0mm 0mm 0mm, clip]{PE0p8_srb111_dtO0p05_dtP2_s10_sd8_r28_rd22p4_b2p6667_bs2p1333_muFac1p8_muP0p0018_delta}
  \caption{\label{subfig:srb08_111}$\widetilde{\sigma}=0.8\sigma$, $\widetilde{\rho}=0.8\rho$,  $\widetilde{\beta}=0.8\beta$}
\end{subfigure}
\caption{\label{fig:srb08}
(Log-linear plots) Parameter recovery with observations every 500 time-steps.  
$\sigma=10$, 
$\rho=28$, 
$\beta=8/3$, , 
$\Delta t=0.0001$,
$\mu^{\text{AOT}}=1.8/\Delta t=1,800$,
$\mu^{\text{param}}=1.8$.
(A) Observations only on $x$,
(B) Observations only on $y$,
(C) Observations only on $y$ and $z$,
(D) Observations only on $x$ and $z$,
(E) Observations only on $x$ and $y$,
(F) Observations on $x$, $y$, and $z$. 
Note: observations only on $z$ with an unknown $\beta$ parameter did not converge and hence are not shown.  In
(A), 
the solution and parameter momentarily converged to the exact value at $t\approx91.2$.
}
\end{figure}

\subsubsection{Additional details of numerical methods}\label{sec:adonm}
For reference, all of the input parameters are given in the MATLAB code in \Cref{sect:matlab}. We describe the basic outline here for clarity. Time-stepping for all algorithms was done via fully explicit Euler time-stepping (see the discussion in \Cref{sec:numerical_methods}).  Parameter updates were done at fixed time intervals, rather than using the analytical conditions provided in the theorems above.  (We regard this as a testament to the robustness of our algorithms; namely, parameter update times do not need to be as precisely determined as our analysis might indicate).  To avoid dividing by zero in the parameter update formulas, a simple \textit{ad-hoc} tolerance value was used; namely, if the denominator was within \texttt{0.0001} of zero, a parameter update would not be performed.  Our results did not seem to depend very strongly on this tolerance value, and nearly identical results were observed even after increasing or decreasing this value by several orders of magnitude.  To initialize our simulations on or near the global attractor of the system, the initial condition in this section are given by the output of the following MATLAB code:
\[
\texttt{ode45(@(t,U)[10*(U(2)-U(1));U(1)*(28-U(3))-U(2);U(1)*U(2)-8/3*U(3)],[0,100],[1,1,1])}
\]
resulting in initial data
\begin{align}\label{eq:ic}
    (x_0,y_0,z_0)=(8.15641407246436,~10.8938717856828,~22.3338694390332).
\end{align}
In our tests, starting with significantly different, randomly generated initial conditions did not yield qualitatively different results; hence we only display results using initial data \eqref{eq:ic}.


\subsection{Multi-parameter recovery with sparse sampling, stochasticity, and noisy observations}\label{sec:mprwssit_noise}
In this section, we demonstrate that our algorithms work with not only unknown parameters and discrete sampling in time, but also with noisy observations and/or stochastic forcing on the underlying equation.  Rigorous analysis of these cases is of fundamental interest, but for brevity is relegated to future work.  The simulations reported in this section are meant to indicate that the parameter learning algorithms developed here are robust to situations of greater physical interest.

For instance, the algorithms are not particularly dependent on having exact observational data, nor continuous-in-time observations, and the underlying model can even have stochastic forcing.  Of course, there is a price to pay for these additional sources of uncertainty.  Exponential convergence rates still occur, but the error reaches a minimum value determined by the magnitude of the noise in the observations (denoted $\eta\geq0$) and the magnitude of the stochastic forcing (denoted $\epsilon\geq0$). The results are shown in \Cref{fig:srb08_noise_part1,fig:srb08_noise_part2}, where observations are sparse-in-time and $\sigma$, $\rho$, and $\beta$ are all unknown. Here, we see that even in the presence of either observational noise (6B) or stochastic forcing (6C), the error still decays exponentially but the minimal error increases.  We then increase both $\epsilon$ and $\eta$ together from $10^{-13}$ (7A) up to $10^{-3}$ (7F) and see that exponential decay persists, but the minimal error correspondingly increases.  It is notable that in \cref{fig:srb08_noise:I}, even with sparse-in-time observations, strong stochastic forcing, strongly noisy observations, and all parameters unknown, our algorithms still converge to within 1--2 decimal places of the right answer.
  
Finally, although we do not show the results here, we note that  longer times between observations were observed to decrease convergence rates as one might expect, although rates remained exponential in all cases we tested until observation times became so separated that no convergence was observed.

\begin{figure}[ht]
\centering
\begin{subfigure}{.32\textwidth}
  \centering
  \includegraphics[width=1.0\linewidth,trim = 0mm 0mm 0mm 0mm, clip]{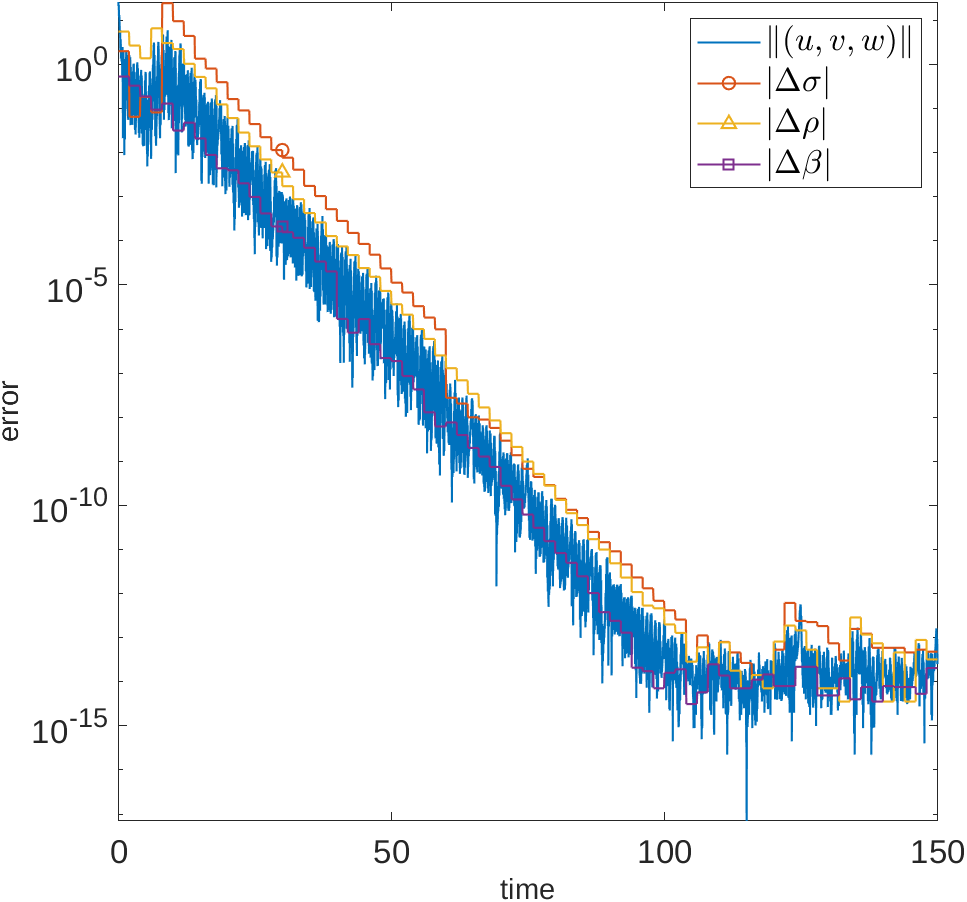}
  \caption{$\epsilon=0$, $\eta=0$ }
  \label{fig:srb08_noise:A}
\end{subfigure}
\begin{subfigure}{.32\textwidth}
  \centering
  \includegraphics[width=1.0\linewidth,trim = 0mm 0mm 0mm 0mm, clip]{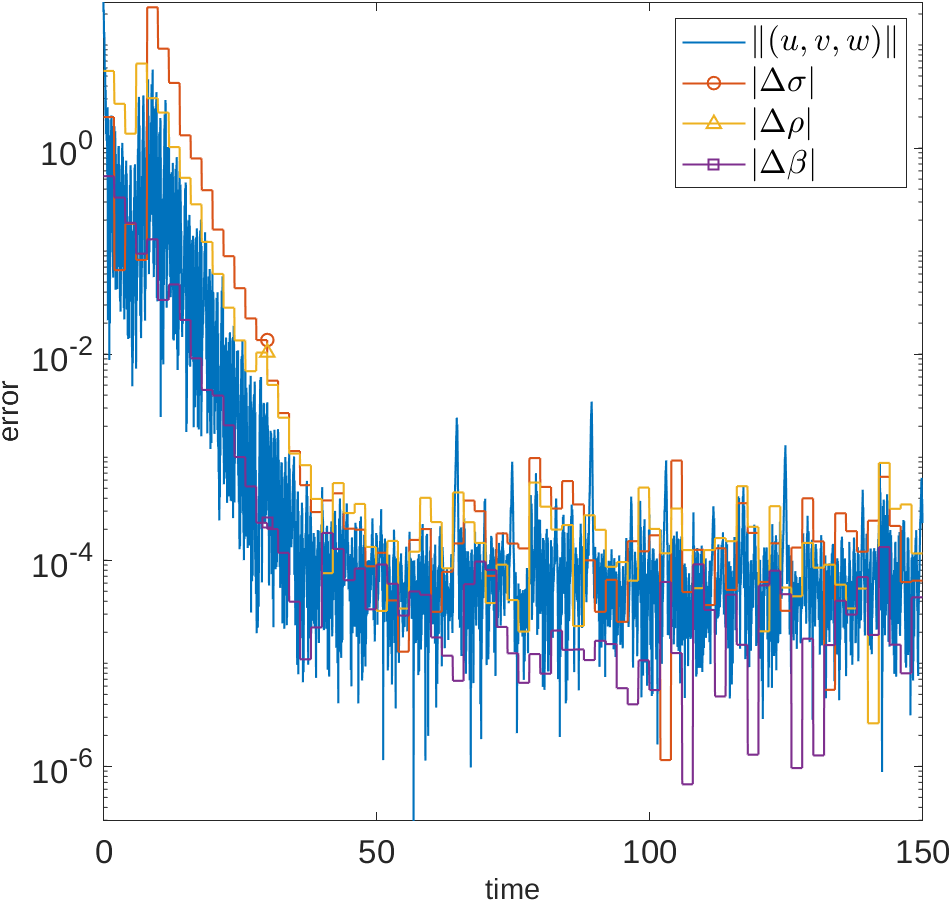}
  \caption{$\epsilon=0$, $\eta=10^{-5}$ }
  \label{fig:srb08_noise:B}
\end{subfigure}
\begin{subfigure}{.32\textwidth}
  \centering
  \includegraphics[width=1.0\linewidth,trim = 0mm 0mm 0mm 0mm, clip]{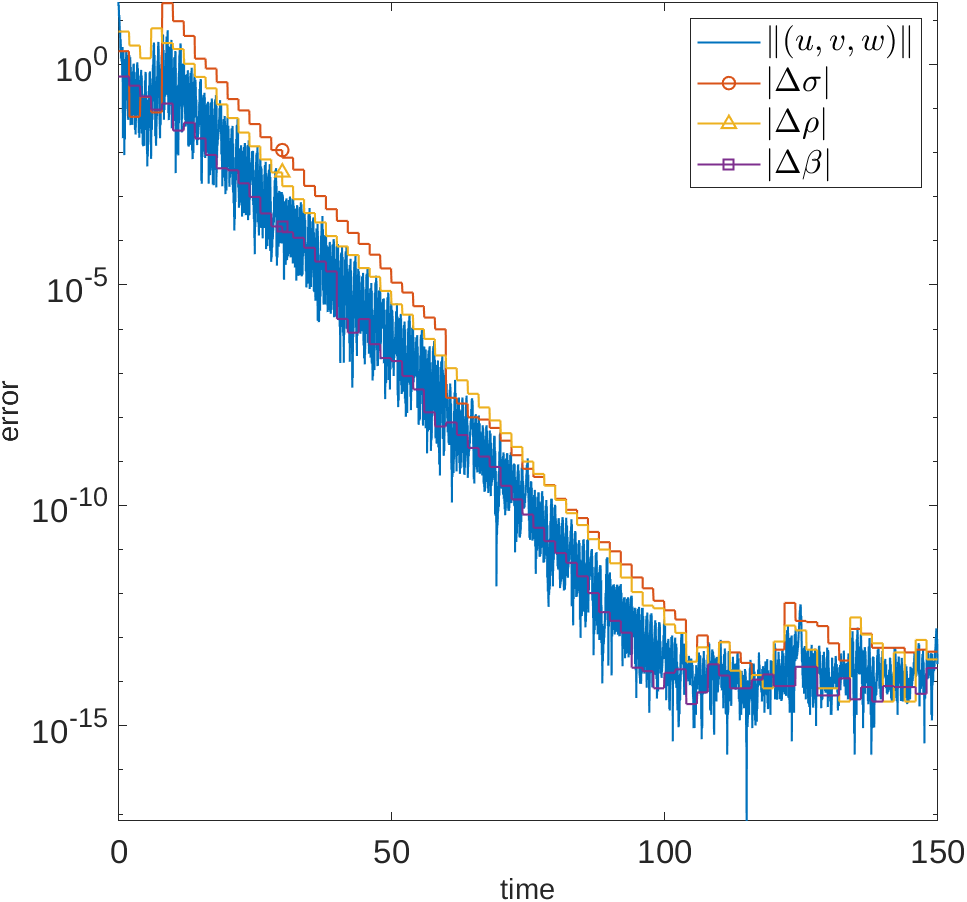}
  \caption{$\epsilon=10^{-5}$, $\eta=0$ }
  \label{fig:srb08_noise:C}
\end{subfigure}
\caption{\label{fig:srb08_noise_part1}
Parameter recovery with observations every 500 time-steps and either stochastic forcing (of amplitude $\epsilon$) or noisy observations (of amplitude $\eta$). The case where neither is present is included for comparison.  
$\sigma=10$, $\widetilde{\sigma}=0.8\sigma$, 
$\rho=28$, $\widetilde{\rho}=0.8\rho$, 
$\beta=8/3$, $\widetilde{\beta}=0.8\beta$,
$\Delta t=0.0001$, 
$\mu^{\text{AOT}}=1.8/\Delta t$,
$\mu^{\text{param}}=1.8$.
}
\end{figure}

\begin{figure}
\begin{subfigure}{.32\textwidth}
  \centering
  \includegraphics[width=1.0\linewidth,trim = 0mm 0mm 0mm 0mm, clip]{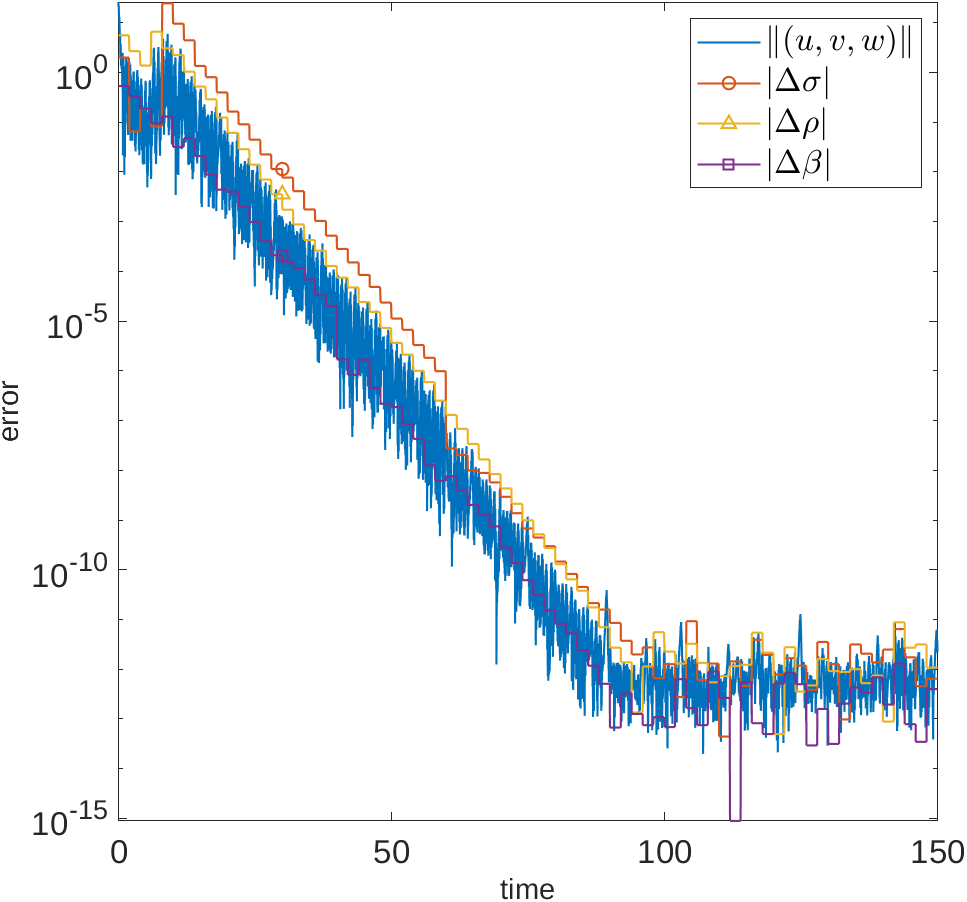}
  \caption{$\epsilon=10^{-13}$, $\eta=10^{-13}$ }
  \label{fig:srb08_noise:D}
\end{subfigure}
\begin{subfigure}{.32\textwidth}
  \centering
  \includegraphics[width=1.0\linewidth,trim = 0mm 0mm 0mm 0mm, clip]{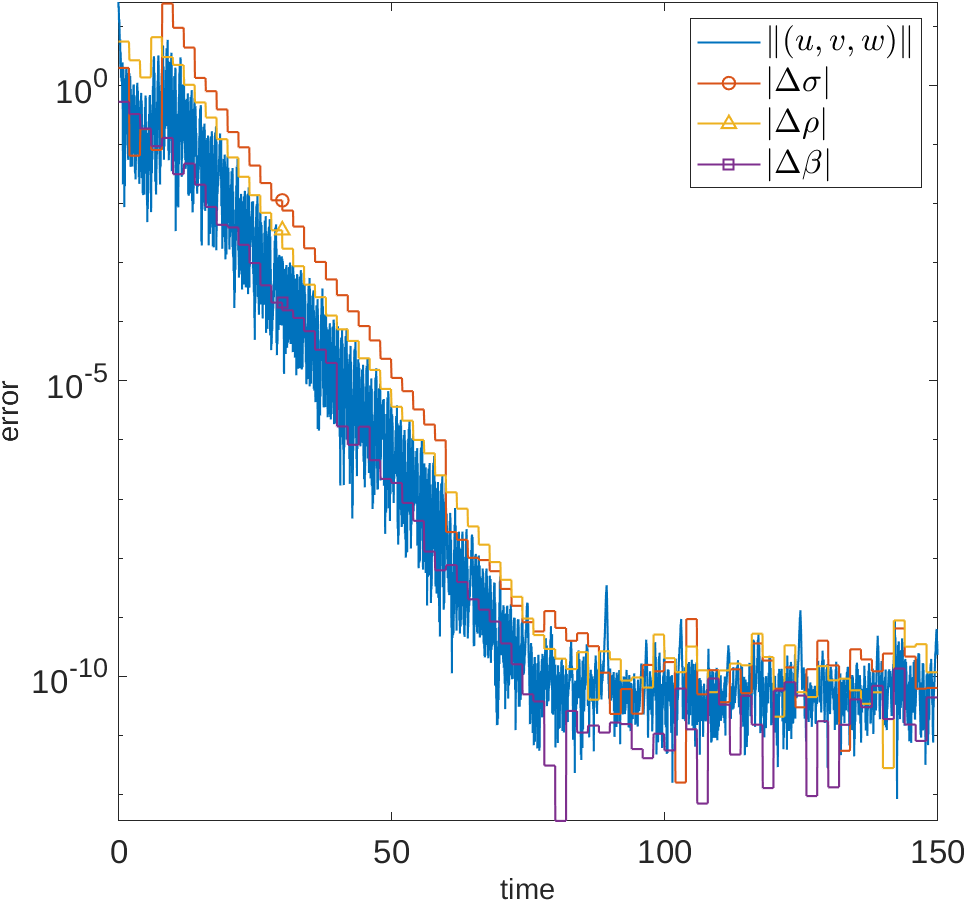}
  \caption{$\epsilon=10^{-11}$, $\eta=10^{-11}$ }
  \label{fig:srb08_noise:E}
\end{subfigure}
\begin{subfigure}{.32\textwidth}
  \centering
  \includegraphics[width=1.0\linewidth,trim = 0mm 0mm 0mm 0mm, clip]{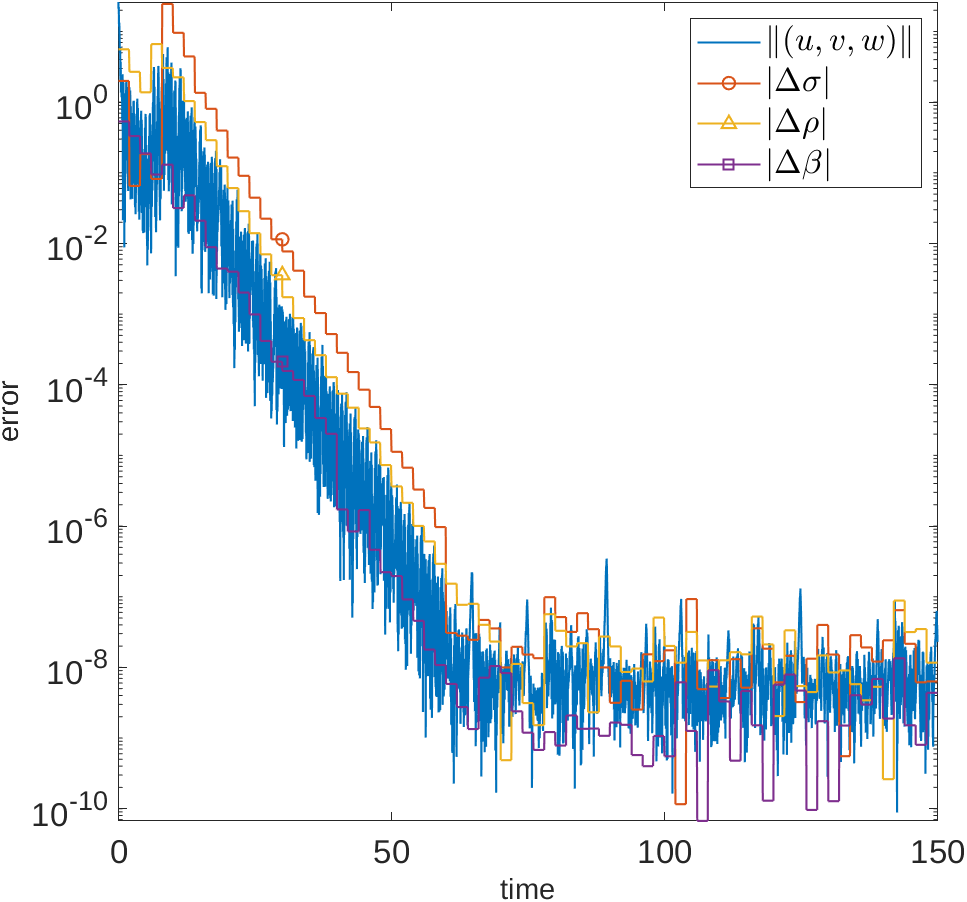}
  \caption{$\epsilon=10^{-9}$, $\eta=10^{-9}$ }
  \label{fig:srb08_noise:F}
\end{subfigure}
\begin{subfigure}{.32\textwidth}
  \centering
  \includegraphics[width=1.0\linewidth,trim = 0mm 0mm 0mm 0mm, clip]{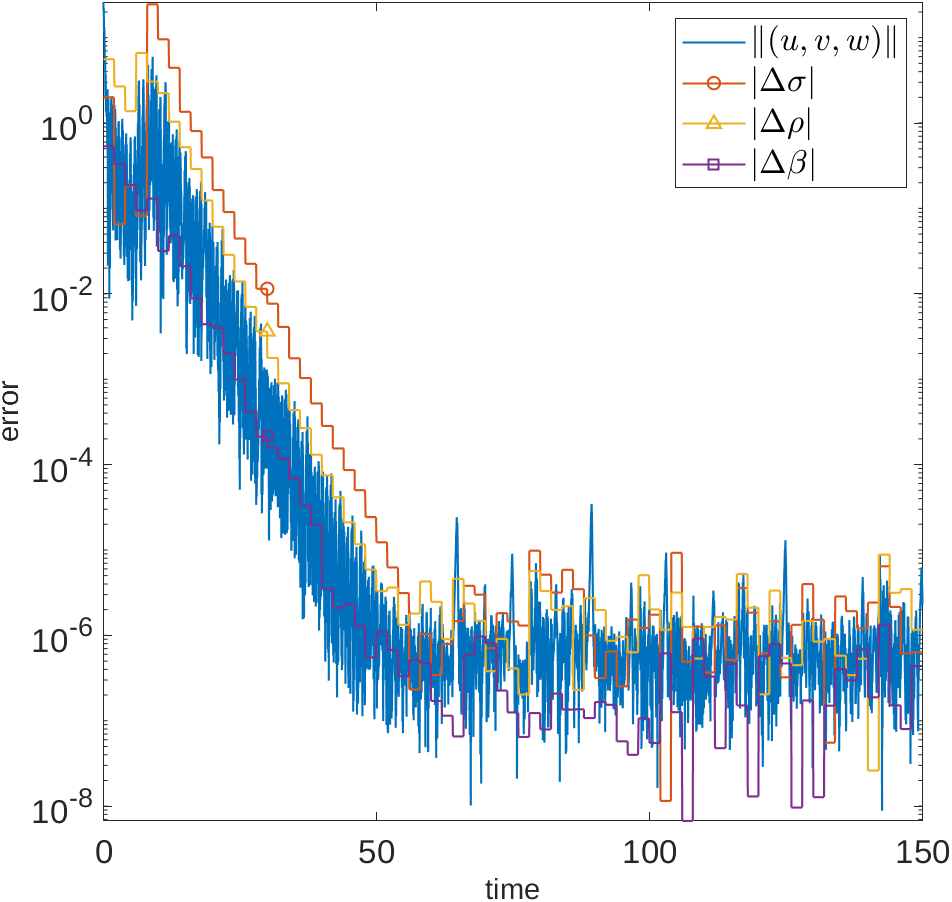}
  \caption{$\epsilon=10^{-7}$, $\eta=10^{-7}$ }
  \label{fig:srb08_noise:G}
\end{subfigure}
\begin{subfigure}{.32\textwidth}
  \centering
  \includegraphics[width=1.0\linewidth,trim = 0mm 0mm 0mm 0mm, clip]{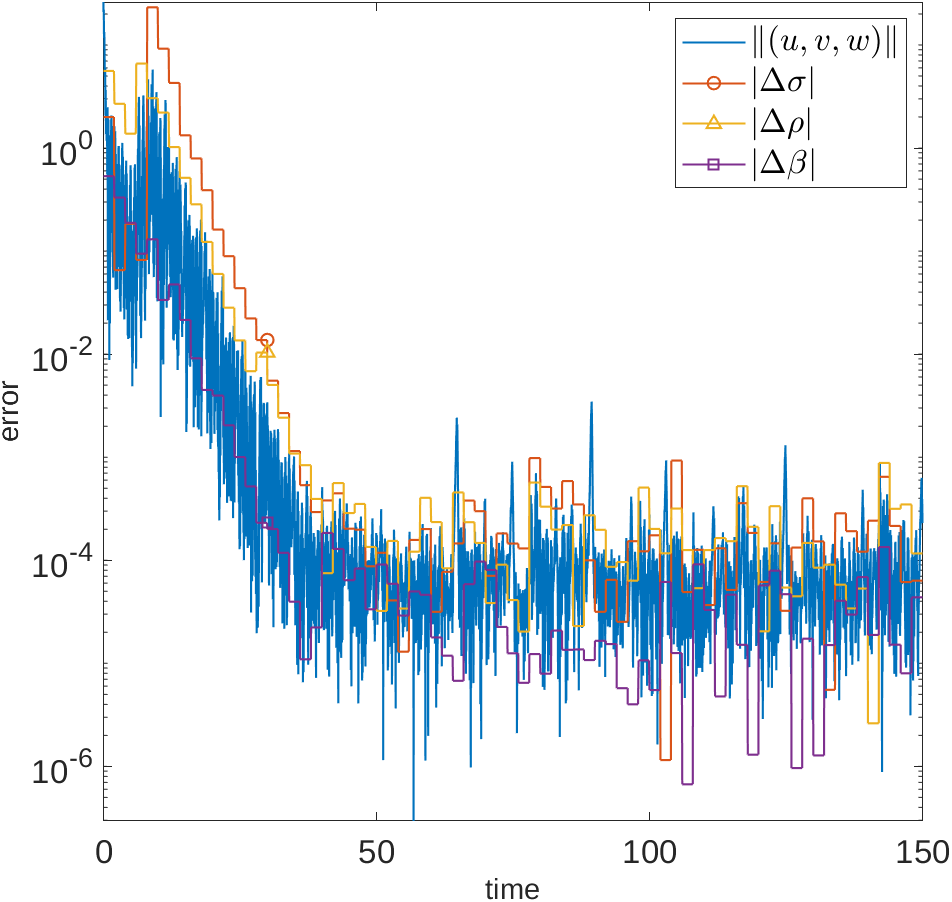}
  \caption{$\epsilon=10^{-5}$, $\eta=10^{-5}$ }
  \label{fig:srb08_noise:H}
\end{subfigure}
\begin{subfigure}{.32\textwidth}
  \centering
  \includegraphics[width=1.0\linewidth,trim = 0mm 0mm 0mm 0mm, clip]{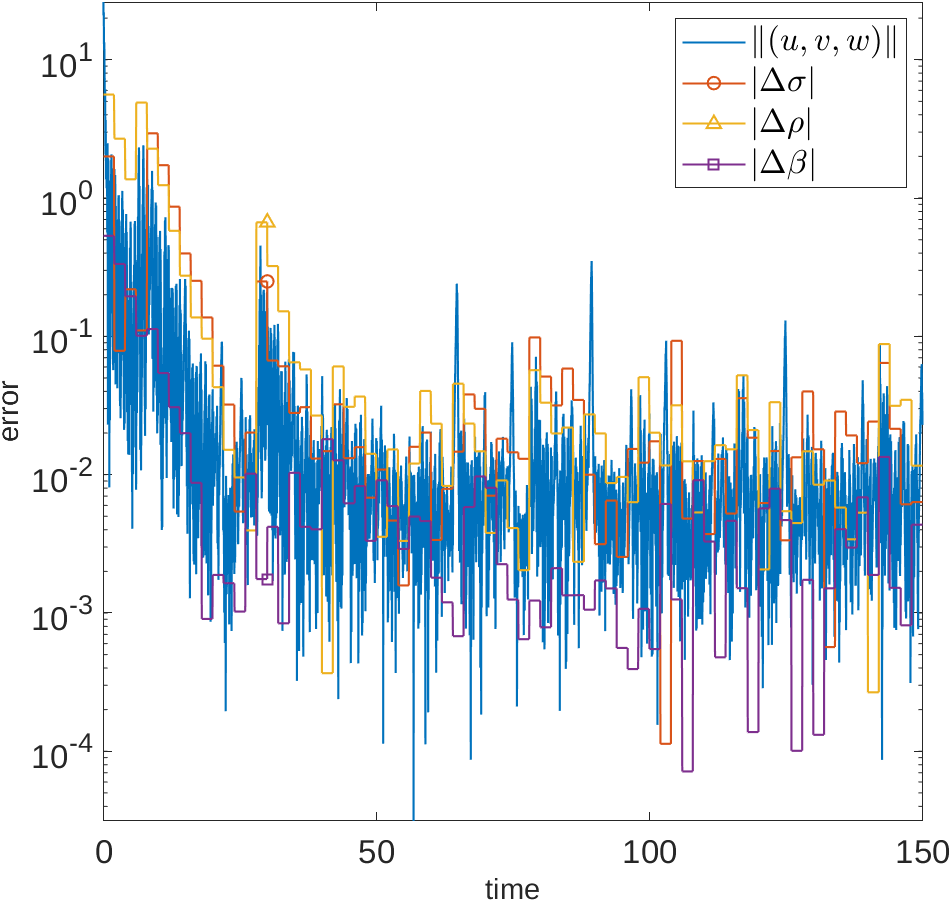}
  \caption{$\epsilon=10^{-3}$, $\eta=10^{-3}$ }
  \label{fig:srb08_noise:I}
\end{subfigure}
\caption{
\label{fig:srb08_noise_part2}Parameter recovery with observations every 500 time-steps and both stochastic forcing (of amplitude $\epsilon$) or noisy observations (of amplitude $\eta$). The initial estimates and algorithm parameters are the same as in \Cref{fig:srb08_noise_part1}.}
\end{figure}

\subsubsection{Additional details of numerical methods} 
Except for the stochastic forcing and noisy observations, the computational setup was identical to that described in \Cref{sec:adonm}.
Again, details can be seen in the MATLAB code in \Cref{sect:matlab}, but we describe the basic implementation here for clarity. 
To implement stochastic forcing of amplitude $\epsilon>0$, a standard white-noise term, i.e., independent Gaussian-normal distributions with mean zero and standard deviation $\epsilon\sqrt{\Delta t}$, were added to the right-hand side of the Lorenz system at each time-step (with seed initialized via \texttt{rng(0)} before the time loop). 
To simulate noisy measurements, the observations of the Lorenz system had Gaussian noise of mean zero and standard deviation $\eta$ added to them before they were used to update/nudge the assimilated solution.

\FloatBarrier
\clearpage

\section{Proofs of main results}\label{sect:proofs}
We now prove the convergence of the parameter learning algorithm under the conditions that were precisely stated in \cref{thm:convergence:x} and \cref{thm:convergence:xy}. To do so, we must first develop a priori estimates, which we do in \cref{sect:apriori}. The proofs of the main theorems are then provided in \cref{sect:proofs:details}.
 
\subsection{A priori estimates}\label{sect:apriori}
First, let us recall the notation \eqref{def:diff}, \eqref{def:par:err}. For our analysis, it will be useful to rewrite \eqref{eq:diff_pre} as
\begin{empheq}[left=\empheqlbrace]{align}\label{eq:diff}\begin{split}
        \dot{u} &=  (y-x)\Delta\sigma+\st v-(\mu_1+\til{\s}) u\\
        \dot{v} &= x\De\rho +\rt u -uw-uz-xw-(1+\mu_2)v \\
        \dot{w} &= -z\De\be  +uv+uy+xv-(\mu_3+\bt)w.
\end{split}\end{empheq}

We then define
\begin{align}\label{def:K:funct}
    \mathcal{K}(u,v,w):= \frac{1}{2}(u^2 + v^2 + w^2),
\end{align}
and 
\begin{align}
\mu&:=\min\left\{\mu_1+\st,\mu_2+1,\mu_3+\bt\right\}\label{def:mu}.
\end{align}
We will first establish the following estimate on $(u,v,w)$.

\begin{Prop}\label{prop:diff:energy}
Let $(x_0,y_0,z_0-\rho-\s)\in\mathscr{B}$, $\mu_1,\mu_2,\mu_3\geq0$, and $\st,\rt,\bt\geq0$. Let $\mu$ be given by \eqref{def:mu}. Suppose that $\mu_1,\mu_2$ satisfy 
\begin{align}\label{cond:mu:diff}
      \mu_1+\st\geq 16\left[\frac{(\De\s)^2+3(\s+\rho)^2+2R^2}{\mu_2+1}+\frac{R^2}{2(\mu_3+\bt)}\right],
    \quad 
      \mu_2+1\geq 4\frac{(\De\rho)^2}{\mu_1+\st}.
\end{align}
Then for $0\leq t_0\leq t$, we have
\begin{align}
    \mathcal{K}(t) & \leq \mathcal{K}(t_0)e^{-\frac{\mu}2(t-t_0)}  +K^2(1-e^{-\frac{\mu}2(t-t_0)}),\notag
\end{align}
where
  \begin{align}
        K^2&:=\frac{2}{\mu}\left\{\frac{2R^2}{\mu_1+\st}(\De\s)^2
    +\frac{4R^2}{\mu_2+1}(\De\rho)^2 
    + \frac{4\left[R^2+(\rho+\s)^2\right]}{\mu_3+\bt}(\De\be)^2\right\}\label{def:K}.
    \end{align}
In particular, if $t\geq t_0+\frac{2}{\mu}\log\left(\mathcal{K}(t_0)/K^2\right)$, then $\mathcal{K}(t)\leq 2K^2$.
\end{Prop}

We immediately deduce the following in the special case $\rt=\rho$, $\bt=\beta$, and $\mu_2=\mu_3=0$.

\begin{Cor}\label{cor:diff:energy}
Let $(x_0,y_0,z_0-\rho-\s)\in\mathscr{B}$ and $\mu$ be given by \eqref{def:mu}. Suppose  $\rt=\rho$, $\bt=\beta$, $\mu_2=0$,$\mu_3=0$,and that $\mu_1 \geq0$ satisfies \eqref{cond:mu:diff}. Then
\begin{align}
    \mathcal{K}(t) & \leq \mathcal{K}(t_0)e^{-\frac{\mu}2(t-t_0)}  +\frac{c_1}{\mu_1+\st}(\De\s)^2,\notag
\end{align}
for some constant $c_1>0$, depending only on $R,\rho,\s,\be$.
\end{Cor}

\begin{proof}[Proof of \cref{prop:diff:energy}]
We calculate
\begin{align}\label{eq:Kdot}
    \dot{\mathcal{K}}
    =& uv\De\s+uv\De\rho+\s uv+\rho uv-zuv+yuw\notag\\
    & +(y-x)u\De\s+xv\De\rho-zw\De\be-(\mu_1+\st)u^2-(1+\mu_2)v^2-(\mu_3+\bt)w^2\notag\\
    =&K_1+K_2+\dots+K_{10}+K_{11}+K_{12}.
\end{align}
We treat $K_1$--$K_9$ with Young's inequality and the absorbing ball bounds for \eqref{eq:Lor} to estimate
\begin{align}
   K_1 &\leq \frac{4}{\mu_2+1}(\De\s)^2u^2+\frac{\mu_2+1}{16}v^2,
   \quad K_2\leq  \frac{\mu_1+\st}4u^2+\frac{(\De\rho)^2}{\mu_1+\st}v^2,
   \quad K_3\leq \frac{4\s^2}{\mu_2+1}u^2+\frac{\mu_2+1}{16}v^2,\notag\\
    K_4&\leq \frac{4\rho^2}{\mu_2+1}u^2+\frac{\mu_2+1}{16}v^2,
    \quad K_5\leq \frac{8\left[R^2+(\rho+\s)^2\right]}{\mu_2+1}u^2+\frac{\mu_2+1}{16}v^2,
    \quad K_6\leq \frac{2R^2}{\mu_3+\bt}u^2+\frac{\mu_3+\bt}8w^2,\notag
\end{align}
and
\begin{align}
 K_7 &\leq \frac{(y-x)^2}{\mu_1+\st}(\De\s)^2+\frac{\mu_1+\st}4u^2\leq \frac{2R^2}{\mu_1+\st}(\De\s)^2+\frac{\mu_1+\st}4u^2,\notag \\
    K_8 &\leq \frac{4x^2}{\mu_2+1}(\De\rho)^2+\frac{\mu_2+1}{16}v^2\leq \frac{4R^2}{\mu_2+1}(\De\rho)^2+\frac{\mu_2+1}{16}v^2, \notag \\
    K_9 &=(z-\rho-\s)w\De\be+(\rho+\s)w\De\be\leq 4\left[(z-\rho-\s)^2+(\rho+\s)^2\right]\frac{(\De\be)^2}{\mu_3+\bt}+ \frac{\mu_3+\bt}8w^2\notag\\
    &\leq 4\left[R^2+(\rho+\s)^2\right]\frac{(\De\be)^2}{\mu_3+\bt}+\frac{\mu_3+\bt}8w^2.\notag
\end{align} 

Upon returning to \eqref{eq:Kdot} and applying the estimates above, we obtain
\begin{align*}
    \dot{\mathcal{K}} \leq& -\frac{1}2\left\{\mu_1+\st-8\left[\frac{(\De\s)^2+3(\s+\rho)^2+2R^2}{\mu_2+1}+\frac{R^2}{2(\mu_3+\bt)}\right]\right\}u^2\notag\\
    &-\frac{1}2\left[\mu_2+1-\frac{2(\De\rho)^2}{\mu_1+\st}\right]v^2
    -\frac{\mu_3+\bt}2w^2 \notag
    \\
    &+\frac{2R^2}{\mu_1+\st}(\De\s)^2
    +\frac{4R^2}{\mu_2+1}(\De\rho)^2 
    + \frac{4\left[R^2+(\rho+\s)^2\right]}{\mu_3+\bt}(\De\be)^2.
\end{align*}
By \eqref{cond:mu:diff}, it follows that
\begin{align}
    \dot{\mathcal{K}} &\leq -\frac{\mu}2\mathcal{K}
    +\frac{2R^2}{\mu_1+\st}(\De\s)^2
    +\frac{4R^2}{\mu_2+1}(\De\rho)^2 
    + \frac{4\left[R^2+(\rho+\s)^2\right]}{\mu_3+\bt}(\De\be)^2.\notag
\end{align}
Hence, Gr\"onwall's inequality yields
\begin{align}
    \mathcal{K}(t) &\leq e^{-\frac{\mu}2(t-t_0)}\mathcal{K}(t_0)+\frac{2}{\mu}\left\{\frac{2R^2}{\mu_1+\st}(\De\s)^2
    +\frac{4R^2}{\mu_2+1}(\De\rho)^2 
    + \frac{4\left[R^2+(\rho+\s)^2\right]}{\mu_3+\bt}(\De\be)^2\right\}(1-e^{-\frac{\mu}2(t-t_0)}),\notag
\end{align}
as desired.
\end{proof}

Now let us denote the derivatives of the differences in \eqref{eq:diff} by
\begin{align}\label{def:diff:dot}
    \gamma := \dot{u}, \quad \delta := \dot{v}, \quad \eta := \dot{w}.
\end{align}
The evolution of $(\gam,\de,\eta)$ is governed by
\begin{empheq}[left=\empheqlbrace]{align}\label{eq:diff:dot}\begin{split}
            \dot{\gam}&=\st\de+(\dot{y}-\dot{x})\De\s-(\mu_1+\st)\gam\\
            \dot{\de}&=\rt\gam+\dot{x}\De\rho-\gam w-u\eta-\gam z-\dot{z}u-\dot{x}w-x\eta-(\mu_2+1)\de\\
            \dot{\eta}&=-\dot{z}\De\be+\gam v+u\de+\gam y+\dot{y}u+\dot{x}v+x\de-(\mu_3+\bt)\eta.
\end{split}\end{empheq}
Similar to \eqref{def:K:funct}, we consider the functional
\begin{align}\label{def:L:funct}
    \mathcal{L}(t) := \frac{1}{2}(\gamma^2+\delta^2+\eta^2).
\end{align}
We prove the following.

\begin{Prop}\label{prop:diff:dot:energy}
 Let $(x_0,y_0,z_0-\rho-\s)\in\mathscr{B}$. Suppose that $\mu_1,\mu_2,\mu_3\geq0$ satisfy \eqref{cond:mu:diff}. Suppose, moreover, that $\mu_1$ satisfies
    \begin{align}\label{cond:mu:diff:dot}
        \mu_1+\st\geq\frac{32\left[(\De\s)^2 + (\De\rho)^2 +R^2+2(\rho+\s)^2\right]}{\mu_2+1}+64\left[\frac{1}{\mu_2+1}+\frac{1}{\mu_3+\bt}\right]K^2+\frac{16R^2}{\mu_3+\bt},
    \end{align}
where $K$ is given by \eqref{def:K}. Then given $t_0>0$
  \begin{align}
      \mathcal{L}(t) &\leq e^{-\frac{\mu}2(t-t_0)}\mathcal{L}(t_0)+L^2,\notag
  \end{align}
provided that $t\geq t_0+\frac{2}{\mu}\log\left(\mathcal{K}(t_0)/K^2\right)$, where $\mu$ is given by \eqref{def:mu}, and $L$ is defined by
    \begin{align}\label{def:L}
        L^2:=&\frac{128}{\mu}\left\{\frac{R^4+(\be^2+\s^2)R^2+\be^2(\rho+\s)^2}{\mu_2+1}+\frac{3R^2\left[R^2+2(\rho+\s)^2+1\right]}{\mu_3+\bt}\right\}K^2\\
    &+\frac{2}{\mu}\left\{\frac{24R^2\left[R^2+(\rho+\s)^2+1\right]}{\mu_1+\st}(\De\s)^2+\frac{8\s^2R^2}{\mu_2+1}(\De\rho)^2+\frac{16\left[R^4+\be^2R^2+\be^2(\rho+\s)^2\right]}{\mu_3+\bt}(\De\be)^2\right\}.\notag
    \end{align}
In particular, if $t$ additionally satisfies $t\geq t_0+\frac{2}{\mu}\log(\mathcal{L}(t_0)/L^2)$, then $\mathcal{L}(t)\leq 2L^2$.
\end{Prop}

As before, we immediately deduce the following in the special case $\rt=\rho$, $\bt=\beta$, and $\mu_2=\mu_3=0$.

\begin{Cor}\label{cor:diff:dot:energy}
Let $(x_0,y_0,z_0-\rho-\s)\in\mathscr{B}$. Suppose  $\rt=\rho$, $\bt=\beta$, $\mu_2=\mu_3=0$, and that $\mu_1\geq0$ satisfies \eqref{cond:mu:diff} and \eqref{cond:mu:diff:dot}. Then
\begin{align}
    \mathcal{L}(t) & \leq \mathcal{L}(t_0)e^{-\frac{\mu}2(t-t_0)}  +\frac{c_2}{\mu_1+\st}(\De\s)^2,\notag
\end{align}
for some constant $c_2>0$, depending only on $R,\rho,\s,\be$.
\end{Cor}

\begin{proof}[Proof of \cref{prop:diff:dot:energy}]
We calculate
\begin{align}\label{eq:Ldot}
    \dot{\mathcal{L}}
    =&(\De\s+\De\rho)\gam\de+(\s+\rho)\gam\de-w\gam\de+v\gam\eta+y\gam\eta-z\gam\de-\dot{z}u\de-\dot{x}w\de+\dot{y}u\eta+\dot{x}v\eta\notag\notag\\
    &+(\dot{y}-\dot{x})\gam\De\s+\dot{x}\de\De\rho-\dot{z}\eta\De\be-(\mu_1+\st)\gam^2-(\mu_2+1)\de^2-(\mu_3+\bt)\eta^2\notag\\
    =&L_1+L_2+\dots+L_{14}+L_{15}+L_{16}.
\end{align}
Before estimating the terms $L_j$, we first collect bounds for $\dot{x},\dot{y},\dot{z}$ by making use of the absorbing ball bounds for \eqref{eq:Lor}. Indeed, we have
    \begin{align}\label{est:Lor:dot}
        \begin{split}
        \dot{x}^2&\leq 2\s^2(x^2+y^2)\leq 2\s^2R^2,\\
        \dot{y}^2&\leq3\left(\rho^2x^2+y^2+2x^2(z-\rho-\s)^2+2x^2(\rho+\s)^2\right)\leq12\left(R^2+(\rho+\s)^2+1\right)R^2,\\
        \dot{z}^2&\leq2\left(x^2y^2+2\be^2(z-\rho-\s)^2+2\be^2(\rho+\s)^2\right)\leq 4\left(R^4+\be^2R^2+\be^2(\rho+\s)^2\right).
        \end{split}
    \end{align}
By assumption, $t$ has been taken sufficiently large so that \cref{prop:diff:energy} guarantees, $\mathcal{K}(t)\leq 2K^2$. We then treat $L_1$--$L_{13}$ with Young's inequality, the absorbing ball bounds for \eqref{eq:Lor}, and this bound to estimate
\begin{align}
   L_1 +L_2&\leq\frac{8\left[(\De\s)^2+(\De\rho)^2+(\s+\rho)^2\right]}{\mu_2+1}\gam^2+\frac{\mu_2+1}{8}\de^2,\notag\\
   L_3 &\leq\frac{4}{\mu_2+1}w^2\gam^2+\frac{\mu_2+1}{16}\de^2\leq\frac{16}{\mu_2+1}K^2\gam^2+\frac{\mu_2+1}{16}\de^2,\notag\\
   L_4&\leq \frac{4}{\mu_3+\bt}v^2\gam^2+\frac{\mu_3+\bt}{16}\eta^2\leq \frac{16}{\mu_3+\bt}K^2\gam^2+\frac{\mu_3+\bt}{16}\eta^2,\notag\\
    L_5&\leq \frac{4}{\mu_3+\bt}y^2\gam^2+\frac{\mu_3+\be}{16}\eta^2\leq \frac{4R^2}{\mu_3+\bt}\gam^2+\frac{\mu_3+\bt}{16}\eta^2,\notag\\
    L_6
    &\leq \frac{8}{\mu_2+1}\left[(z-\rho-\s)^2+(\rho+\s)^2\right]\gam^2+\frac{\mu_2+1}{16}\de^2\leq \frac{8\left[R^2+(\rho+\s)^2\right]}{\mu_2+1}\gam^2+\frac{\mu_2+1}{16}\de^2.\notag
\end{align}
Moreover, making use of \eqref{est:Lor:dot}, we obtain the following estimates for $L_7$--$L_{10}$
\begin{align}
    L_7&\leq\frac{4\dot{z}^2}{\mu_2+1}u^2+\frac{\mu_2+1}{16}\de^2\leq\frac{64\left[R^4+\be^2R^2+\be^2(\rho+\s)^2\right]}{\mu_2+1}K^2+\frac{\mu_2+1}{16}\de^2,\notag\\
    L_8&\leq \frac{4\dot{x}^2}{\mu_2+1}w^2+\frac{\mu_2+1}{16}\de^2\leq\frac{32\s^2R^2}{\mu_2+1}K^2+\frac{\mu_2+1}{16}\de^2,\notag\\
    L_9&\leq\frac{4\dot{y}^2}{\mu_3+\bt}u^2+\frac{\mu_3+\bt}{16}\eta^2\leq \frac{192R^2\left[R^2+(\rho+\s)^2+1\right]}{\mu_3+\bt}K^2+\frac{\mu_3+\bt}{16}\eta^2,\notag\\
    L_{10}&\leq \frac{4}{\mu_3+\bt}\dot{x}^2v^2+\frac{\mu_3+\bt}{16}\eta^2\leq \frac{32\s^2R^2}{\mu_3+\bt}K^2+\frac{\mu_3+\bt}{16}\eta^2.\notag
\end{align}
Lastly, we treat $L_{11}$--$L_{13}$ similar to above and obtain
\begin{align}
   L_{11}&\leq\frac{\mu_1+\st}{2}\gam^2+\frac{|\dot{y}-\dot{x}|^2}{2(\mu_1+\st)}(\De\s)^2\leq\frac{\mu_1+\st}2\gam^2+\frac{24R^2\left[R^2+(\rho+\s)^2+1\right]}{\mu_1+\st}(\De\s)^2,\notag\\
   L_{12}&\leq\frac{4}{\mu_2+1}\dot{x}^2(\De\rho)^2+\frac{\mu_2+1}{16}\de^2\leq \frac{8\s^2R^2}{\mu_2+1}(\De\rho)^2+\frac{\mu_2+1}{16}\de^2,\notag\\
   L_{13}&\leq \frac{4}{\mu_3+\bt}\dot{z}^2(\De\be)^2+\frac{\mu_3+\bt}{16}\eta^2\leq \frac{16\left[R^4+\be^2R^2+\be^2(\rho+\s)^2\right]}{\mu_3+\bt}(\De\be)^2+\frac{\mu_3+\bt}{16}\eta^2\notag.
\end{align}
Combining these, we arrive at
\begin{align}
    \dot{\mathcal{L}}\leq&-\frac{1}{2}\left\{\mu_1+\st-\frac{16\left[(\De\s)^2 + (\De\rho)^2 +R^2+2(\rho+\s)^2\right]}{\mu_2+1}-32\left[\frac{1}{\mu_2+1}+\frac{1}{\mu_3+\bt}\right]K^2-\frac{8R^2}{\mu_3+\bt}\right\}\gam^2\notag
    \\
    &-\frac{\mu_2+1}2\de^2-\frac{\mu_3+\bt}2\eta^2\notag
    \\
    &+64\left\{\frac{R^4+(\be^2+\s^2)R^2+\be^2(\rho+\s)^2}{\mu_2+1}+\frac{3R^2\left[R^2+2(\rho+\s)^2+1\right]}{\mu_3+\bt}\right\}K^2\notag
    \\
    &+\frac{24R^2\left[R^2+(\rho+\s)^2+1\right]}{\mu_1+\st}(\De\s)^2+\frac{8\s^2R^2}{\mu_2+1}(\De\rho)^2+\frac{16\left[R^4+\be^2R^2+\be^2(\rho+\s)^2\right]}{\mu_3+\bt}(\De\be)^2\notag.
\end{align}
By \eqref{cond:mu:diff:dot} and Gr\"onwall's inequality, it follows that
\begin{align}
    \mathcal{L}(t)\leq& e^{-\frac{\mu}2(t-t_0)}\mathcal{L}(t_0)+\frac{128}{\mu}\left\{\frac{R^4+(\be^2+\s^2)R^2+\be^2(\rho+\s)^2}{\mu_2+1}+\frac{3R^2\left[R^2+2(\rho+\s)^2+1\right]}{\mu_3+\bt}\right\}K^2\notag\\
    &+\frac{2}{\mu}\left\{\frac{24R^2\left[R^2+(\rho+\s)^2+1\right]}{\mu_1+\st}(\De\s)^2+\frac{8\s^2R^2}{\mu_2+1}(\De\rho)^2+\frac{16\left[R^4+\be^2R^2+\be^2(\rho+\s)^2\right]}{\mu_3+\bt}(\De\be)^2\right\},\notag
\end{align}
as desired.
\end{proof}

Recall the notation introduced in \eqref{def:diff:dot} and define the functionals
    \begin{align}\label{def:dot:funct}
    \mathcal{G}(\gam): = \frac{1}{2}\gamma^2,\quad \mathcal{D}(\de): =\frac{1}2\de^2,\quad\mathcal{E}(\eta):=\frac{1}2\eta^2.
    \end{align}
\begin{Prop}\label{prop:remainder}
 Let $(x_0,y_0,z_0-\rho-\s)\in\mathscr{B}$. Suppose that $\mu_1,\mu_2,\mu_3\geq0$ satisfy \eqref{cond:mu:diff}, \eqref{cond:mu:diff:dot}, and let $\mu$ be given by \eqref{def:mu}. Then given $t_0>0$, it holds that
  \begin{align}\notag
    \mathcal{G}(t)&\leq e^{-(\mu_1+\st)(t-t_0)}\mathcal{G}(t_0)+G^2,\notag\\
    \mathcal{D}(t)&\leq e^{-(\mu_2+1)(t-t_0)}\mathcal{D}(t_0)+D^2,\notag\\
    \mathcal{E}(t)&\leq e^{-(\mu_3+\bt)(t-t_0)}\mathcal{E}(t_0)+E^2, \notag
  \end{align}
provided that $t\geq t_0+\frac{2}{\mu}\log\left(\mathcal{K}(t_0)\mathcal{L}(t_0)/(KL)^2\right)$, where $K,L$ are given by \eqref{def:K}, \eqref{def:L}, respectively, and
    \begin{align}
        G^2&:=\frac{24R^2\left[R^2+2(\rho+\s)^2+1\right]}{(\mu_1+\st)^2}(\De\s)^2+\frac{4\st^2}{(\mu_1+\st)^2}L^2,\label{def:G}\\
       D^2&:=\frac{8\s^2R^2}{\mu_2+1}(\De\rho)^2
        +\frac{128}{\mu_2+1}K^2L^2+\frac{48(R^2+(\rho+\s)^2)}{\mu_2+1}L^2+\frac{128\left[R^4+(\be^2+\s^2)R^2+\be^2(\rho+\s)^2\right]}{\mu_2+1}K^2,\label{def:D}\\
        E^2&:=\frac{8\left[R^4+\be^2R^2+\be^2(\rho+\s)^2\right]}{\mu_3+\bt}(\De\be)^2
        +\frac{64}{\mu_3+\bt}K^2L^2+\frac{16R^2}{\mu_3+\bt}L^2+\frac{48\left[R^2+2(\rho+\s)^2+1\right] R^4}{\mu_3+\bt}K^2.\label{def:E}
    \end{align}
\end{Prop}

\begin{proof}
We calculate
    \begin{align}
        \dot{\mathcal{G}}&=\st\de\gam+(\dot{y}-\dot{x})\gam\De\s-(\mu_1+\st)\gam^2=G_1+G_2+G_3,\notag\\
        \dot{\mathcal{D}}&=\rt\gam\de+\dot{x}\de\De\rho-(\gam w+u\eta)\de-(\gam z+x\eta)\de-(\dot{z}u+\dot{x}w)\de-(\mu_2+1)\de^2=D_1+\dots+D_6,\notag\\
        \dot{\mathcal{E}}&=-\dot{z}\eta\De\be+(v\gam+u\de)\eta+(y\gam+x\de)\eta+(\dot{y}u+\dot{x}v)\eta-(\mu_3+\bt)\eta^2=E_1+\dots+E_5.\notag
    \end{align}
By assumption, $t$ has been taken sufficiently large, so that \cref{prop:diff:dot:energy} guarantees $\mathcal{L}(t)\leq 2L^2$. By Young's inequality and \eqref{est:Lor:dot}, we estimate
    \begin{align}
        G_1&\leq \frac{\mu_1+\st}4\gam^2+\frac{\st^2}{\mu_1+\st}\de^2\leq\frac{\mu_1+\st}4\gam^2+\frac{4\st^2}{\mu_1+\st}L^2,\notag\\
        G_2&\leq \frac{2(\dot{y}^2+\dot{x}^2)}{\mu_1+\st}(\De\s)^2+\frac{\mu_1+\st}4\gam^2\leq \frac{24R^2\left[R^2+2(\rho+\s)^2+1\right]}{\mu_1+\st}(\De\s)^2+\frac{\mu_1+\st}4\gam^2.\notag
    \end{align}
Hence
    \begin{align}
        \dot{\mathcal{G}}\leq -(\mu_1+\st)\mathcal{G}+\frac{24R^2\left[R^2+2(\rho+\s)^2+1\right]}{\mu_1+\st}(\De\s)^2+\frac{4\st^2}{\mu_1+\st}L^2,\notag
    \end{align}
so that by Gr\"onwall's inequality, we deduce
    \begin{align}
        \mathcal{G}(t)\leq e^{-(\mu_1+\st)(t-t_0)}\mathcal{G}(t_0)+\frac{24R^2\left[R^2+2(\rho+\s)^2+1\right]}{(\mu_1+\st)^2}(\De\s)^2+\frac{4\st^2}{(\mu_1+\st)^2}L^2.\notag
    \end{align} 

Similarly, for $\mathcal{D}$, we have
    \begin{align}
        D_1&\leq \frac{4\rt}{\mu_2+1}\gam^2+\frac{\mu_2+1}{16}\de^2\leq \frac{16\rt}{\mu_2+1}L^2+\frac{\mu_2+1}{16}\de^2\notag\\
        D_2&\leq \frac{4\dot{x}^2}{\mu_2+1}+\frac{\mu_2+1}{16}(\De\rho)^2\leq \frac{8\s^2R^2}{\mu_2+1}(\De\rho)^2+\frac{\mu_2+1}{16}\de^2\notag\\
        D_3&\leq \frac{4(w^2\gam^2+u^2\eta^2)}{\mu_2+1}+\frac{\mu_2+1}{8}\de^2\leq \frac{128}{\mu_2+1}K^2L^ 2+\frac{\mu_2+1}{8}\de^2,\notag\\
        D_4&\leq \frac{4(2\gam^2(z-\rho-\sigma)^2 + 2\gam^2(\rho+\sigma)^2+x^2\eta^2)}{\mu_2+1}+\frac{\mu_2+1}8\de^2\leq \frac{48(R^2+(\rho+\sigma)^2)}{\mu_2+1}L^2+\frac{\mu_2+1}8\de^2,\notag\\
        D_5&\leq 8\frac{u^2\dot{z}^2+\dot{x}^2w^2}{\mu_2+1}+\frac{\mu_2+1}{16}\de^2\leq \frac{128\left[R^4+(\be^2+\s^2)R^2+\be^2(\rho+\s)^2\right]}{\mu_2+1}K^2+\frac{\mu_2+1}{16}\de^2.\notag
    \end{align}
Hence
    \begin{align}
        \dot{\mathcal{D}}\leq& -(\mu_2+1)\mathcal{D}+\frac{8\s^2R^2}{\mu_2+1}(\De\rho)^2\notag\\
        &+\frac{128}{\mu_2+1}K^2L^2+\frac{48(R^2+(\rho+\s)^2)}{\mu_2+1}L^2+\frac{128\left[R^4+(\be^2+\s^2)R^2+\be^2(\rho+\s)^2\right]}{\mu_2+1}K^2,\notag
    \end{align}
so that
    \begin{align}
        \mathcal{D}(t)\leq& e^{-(\mu_2+1)(t-t_0)}\mathcal{D}(t_0)+\frac{8\s^2R^2}{\mu_2+1}(\De\rho)^2\notag\\
        &+\frac{128}{\mu_2+1}K^2L^2+\frac{48(R^2+(\rho+\s)^2)}{\mu_2+1}L^2+\frac{128\left[R^4+(\be^2+\s^2)R^2+\be^2(\rho+\s)^2\right]}{\mu_2+1}K^2.\notag
    \end{align}
    
Lastly, we treat $\mathcal{E}$ and estimate
    \begin{align}
        E_1&\leq \frac{2\dot{z}^2}{\mu_3+\bt}(\De\be)^2+\frac{\mu_3+\bt}{8}\eta^2\leq\frac{8\left[R^4+\be^2R^2+\be^2(\rho+\s)^2\right]}{\mu_3+\bt}(\De\be)^2+\frac{\mu_3+\bt}{8}\eta^2,\notag\\
        E_2&\leq \frac{4(v^2\gam^2+u^2\de^2)}{\mu_3+\bt}+\frac{\mu_3+\bt}{8}\eta^2\leq\frac{64}{\mu_3+\bt}K^2L^2+\frac{\mu_3+\bt}{8}\eta^2,\notag\\
        E_3&\leq \frac{4(y^2\gam^2+x^2\de^2)}{\mu_3+\bt}+\frac{\mu_3+\bt}{8}\eta^2\leq \frac{16R^2}{\mu_3+\bt}L^2+\frac{\mu_3+\bt}8\eta^2,\notag\\
        E_4&\leq \frac{4(\dot{y}^2u^2+\dot{x}^2v^2)}{\mu_3+\bt}+\frac{\mu_3+\bt}8\eta^2\leq \frac{48\left[R^2+2(\rho+\s)^2+1\right] R^4}{\mu_3+\bt}K^2+\frac{\mu_3+\bt}8\eta^2.\notag
    \end{align}
It follows that
    \begin{align}
        \dot{\mathcal{E}}\leq& -(\mu_3+\bt)\mathcal{E}+\frac{8\left[R^4+\be^2R^2+\be^2(\rho+\s)^2\right]}{\mu_3+\bt}(\De\be)^2\notag\\
        &+\frac{64}{\mu_3+\bt}K^2L^2+\frac{16R^2}{\mu_3+\bt}L^2+\frac{48\left[R^2+2(\rho+\s)^2+1\right] R^4}{\mu_3+\bt}K^2.\notag
    \end{align}
An application of Gr\"onwall's inequality yields
    \begin{align}
        \mathcal{E}(t)\leq& e^{-(\mu_3+\bt)(t-t_0)}\mathcal{E}(t_0)+\frac{8\left[R^4+\be^2R^2+\be^2(\rho+\s)^2\right]}{\mu_3+\bt}(\De\be)^2\notag\\
        &+\frac{64}{\mu_3+\bt}K^2L^2+\frac{16R^2}{\mu_3+\bt}L^2+\frac{48\left[R^2+2(\rho+\s)^2+1\right] R^4}{\mu_3+\bt}K^2,\notag
    \end{align}
which completes the proof.
\end{proof}

\subsection{Proof of Theorem \ref{thm:convergence:x}}\label{sect:proofs:details}
Recall that we will employ the rule given by \eqref{def:update} for updating values of the unknown parameters $\s,\rho,\be$. From \eqref{eq:diff_pre}, we thus observe that over the interval $I_n$, the derivative of the differences, $\dot{u}$, can be rewritten as
\begin{align}\label{eq:u:proof}
    \dot{u}= (\sigma_n-\sigma)(\ty-\tx) + \sigma(v-u) - \mu_1 u,
\end{align}
where we recall $(u,v,w)$ to be defined by \eqref{def:diff}. We then rearrange this to obtain
\begin{align}\label{eq:sigma:proof}
    \sigma_n-\sigma = \frac{\dot{u}-\s(v-u)+\mu_1 u}{\ty-\tx}.
\end{align}
Upon evaluating at $t=t_{n+1}^-$ and recalling the convention \eqref{def:discrete:diff}, substitution into the parameter recovery formulas \eqref{def:update} then yield the following identities:
\begin{align}
    \begin{split}\label{eq:sigma:diff:proof}
  \sigma_{n+1}-\sigma = \frac{\dot{u}_{n+1}-\sigma(v_{n+1}-u_{n+1})}{\ty_{n+1}-\tx_{n+1}}.
 \end{split}
\end{align}

\begin{proof}[Proof of \cref{thm:convergence:x}]
We proceed by induction. Let $N=1$ and define 
    \begin{align}\notag
        \mu_0=\min\left\{\mu_1,\mu_2+1,\mu_3+\be\right\}.
    \end{align}
Consider any $\s_0>0$ such that $|\s_0-\s|\leq M$. Observe that $\mu\geq\mu_0$, where $\mu$ is given by \eqref{def:mu}. Suppose that $\mu_1$ satisfies \eqref{cond:mu:diff} and \eqref{cond:mu:diff:dot}, where $\De\s$ is replaced by $M$, so that \eqref{cond:mu:large} is satisfied. Let $t_0=0$ and suppose also that
    \begin{align}\notag
        \tau_1> \frac{4}{\mu_0}\log_+\left(\mathcal{K}(0)\mathcal{L}(0)/(KL)^2\right),
    \end{align}
where $\mathcal{K}, \mathcal{L}$ are defined by \eqref{def:K:funct}, \eqref{def:L:funct}, respectively, and all quantities involving $\mu$ are replaced by $\mu_0$ therein. Observe that $\tau_1> t_0$. Now by \cref{cor:diff:energy} and \cref{cor:diff:dot:energy}, we have
    \begin{align}\notag
        \mathcal{K}(t)\leq \frac{c_1}{\mu_1}|\s_0-\s|^2,\quad \mathcal{L}(t)\leq \frac{c_2}{\mu_1}|\s_0-\s|^2,
    \end{align}
for all $0\leq t\leq \tau_1$.

Choose any $t_1$ such that $\tau_1/2< t_1< \tau_1$. Upon returning to \eqref{eq:sigma:diff:proof}, it now follows that
    \begin{align}\notag
        |\s_{n+1}-\s|\leq \frac{2(\sqrt{c_1}+\sqrt{c_2})}{\sqrt{\mu_1}}\frac{|\s_0-\s|}{\veps},
    \end{align}
where we have applied the hypothesis \eqref{cond:non:degen}. By choosing $\mu_1$ sufficiently large satisfying \eqref{cond:mu:converge}, we obtain
    \begin{align}\notag
          |\s_{n+1}-\s|\leq \eps|\s_0-\s|,
    \end{align}
where $\eps=(\s\wedge1)/(2M)\leq1/2$. Notice that this implies $\s_{n+1}>0$. This establishes the base case. 

Now suppose that \eqref{eq:converge} holds for $N>1$. Choose $\mu_1$ as in the base case and $\tau_{N+1}$ such that
    \begin{align}\notag
        \tau_{N+1}> t_N+\frac{4}{\mu}\log_+\left(\mathcal{K}(0)\mathcal{L}(0)/(KL)^2\right),
    \end{align}
so that $\tau_{N+1}>t_N$. Then by \cref{cor:diff:energy} and \cref{cor:diff:dot:energy}, it follows that
    \begin{align}\notag
     \mathcal{K}(t)\leq \frac{c_1}{\mu_1}|\s_N-\s|^2,\quad \mathcal{L}(t)\leq \frac{c_2}{\mu_1}|\s_N-\s|^2,
    \end{align}
for all $t_N\leq t\leq \tau_N$. We now choose any $t_{N+1}$ such that $(t_N+\tau_{N+1})/2< t_{N+1}<\tau_{N+1}$. Then from \eqref{eq:sigma:diff:proof} and the induction hypothesis, we obtain
    \begin{align}\notag
    |\s_{N+1}-\s|\leq \eps|\s_N-\s|\leq \eps^{N+1}|\s_0-\s|.
    \end{align}
By definition of $\eps$, we again have $\s_{N+1}>0$. This completes the proof.
\end{proof}

\begin{Rmk}\label{rmk:multi:proof}
A similar argument to the one presented above for \cref{thm:convergence:x} can also be provided for the proof of \cref{thm:convergence:xy} and, in fact, all other combinations. With slight modifications to the estimates made in \cref{prop:diff:energy}, \cref{prop:diff:dot:energy}, and \cref{prop:remainder}, we may obtain statements analogous to \cref{cor:diff:energy} and \cref{cor:diff:dot:energy}, which are adapted to the case of the particular combination of interest, e.g., recovering $(\s,\rho)$, etc. Indeed, the apriori estimates in \cref{sect:apriori} have been performed for all variables precisely to accommodate all possible combinations for parameter recovery. For these other combinations, in addition to \eqref{eq:u:proof}, one considers
    \begin{align}
         \dot{v}&=(\rho_n-\rho)\tx+\rho u-uw-uz-xw-(1+\mu_2)v\notag\\
    \dot{w}&=-(\be_n-\be)\tz+uv+uy+xv-(\mu_3+\be)w\notag.
    \end{align}
One then derives identities analogous to \eqref{eq:sigma:proof} for $\rho_n$ and $\be_n$ given by
    \begin{align}
        \rho_n-\rho&=\frac{\dot{v}-\rho u+uw+uz+xw+(1+\mu_2)v}{\tx},\notag\\
    \be_n-\be&=\frac{-\dot{w}+\beta w+ uv+uy+xv-\mu_3w}{\tz}\notag.
    \end{align}
Ultimately, one then considers
    \begin{align}
         \rho_{n+1} - \rho &= \frac{\dot{v}_{n+1} -\rho u_{n+1} + u_{n+1}w_{n+1}+ u_{n+1}z_{n+1}+ x_{n+1}w_{n+1} + v_{n+1}}{\tx},\notag
  \\
  \beta_{n+1} - \beta &= \frac{-\dot{w}_{n+1} +u_{n+1}v_{n+1}+u_{n+1}y_{n+1}+x_{n+1}v_{n+1}-\beta w_{n+1}}{\tz},\notag
    \end{align}
which play roles analogous to the one played by \eqref{eq:sigma:diff:proof} in the proof of \cref{thm:convergence:x}. The only major technical differences are that the time derivatives are to be estimated with \cref{prop:remainder} and treatment of the terms which are quadratic in the difference variables, $u,v,w$. The quadratic terms, however, present no difficulties whatsoever, as the estimates supplied in \cref{sect:apriori} ultimately provide sufficient control over these terms. Specifically, our estimates allow them to be treated in such a way as if they were linear in the difference variables to begin with; from this point, the proof then proceeds as in the one provided for \cref{thm:convergence:x} above.
\end{Rmk}

\section{Conclusion}\label{sec:conclusions}
We have developed a rigorously justified algorithm for learning the parameters of the Lorenz equations from partially observed data. Sufficient hypotheses for establishing convergence of this scheme are detailed in Theorem \ref{thm:convergence:x} and \ref{thm:convergence:xy}.  We emphasize that these hypotheses appear to be analytically pessimistic in the sense that the true parameters can be computationally learned using nudging parameters that are much smaller than those indicated by \eqref{cond:mu:large:xy}.  In addition, although the proofs that establish these theorems rely on continuous observations of the state variables, computationally we find that infrequent discrete observations are sufficient.  We also find that the parameter learning algorithm presented here is robust to the presence of stochastic noise in the state evolution, as well as noise in the observation operator.

The approach taken here is sufficiently robust and concise that we anticipate the current results can be extended to other systems more complicated than the Lorenz equations.  For instance, we anticipate that rigorous justification for the learning of multiple parameters in the PDE setting is within reach using the analysis here as a guide.

\newpage
\appendix
\section{Simplified MATLAB code}\label{sect:matlab}
\noindent We present a compressed version of the MATLAB code used to produce the plots in \Cref{sec:mprwssit,sec:mprwssit_noise}.

\begin{lstlisting}[numbers=left]
sigma = 10; rho = 28; beta = 8/3; % Lorenz parameters
sigma_DA=0.8*sigma;rho_DA=0.8*rho;beta_DA=0.8*beta; % Initial guess
t0 = 0; tf = 150; % Initial and final times
dt = 0.0001; % Time-step % Note: ode45 predicts min(dt)=0.00193
dt_obs = 0.05; % How often solution is observed
dt_param = 2.0; % How often to update the parameters
p_tol = 0.0001; % Tolerence for parameter switching
mu = 1.8/dt; % AOT nudging parameter
mu_p = mu/1000; % mu for updating the parameters
eta = 0; % Amplitude of noise of measurements
epsilon = 0; % Stochastic forcing amplitude on Lorenz
t = t0:dt:tf; N = length(t); rng(0);
U=[8.15641407246436;10.8938717856828;22.3338694390332]; V=[0;0;0];
lorenz  = @(U,sto)[    sigma*(U(2)-U(1));U(1)*(rho   -U(3))-U(2);...
           U(1)*U(2)-beta*U(3)] + sto;
lorenz_DA = @(V,FC)[sigma_DA*(V(2)-V(1));V(1)*(rho_DA-V(3))-V(2);...
           V(1)*V(2)-beta_DA*V(3)] - FC; % Data Assimilation
e_sol=zeros(1,N);e_sigma=zeros(1,N);e_rho=zeros(1,N);e_beta=zeros(1,N);
obs_int = round(dt_obs/dt); param_int = round(dt_param/dt);
for ti = 1:N
  e_sol(ti) = norm(U-V); e_sigma(ti) = abs(sigma_DA-sigma);
  e_rho(ti) = abs(rho_DA-rho); e_beta(ti) = abs(beta_DA-beta);
  if mod(ti,obs_int) == 1 % If true, update feedback control term
    FC = mu*(V - (U + eta*randn(3,1)));
  else
    FC = [0;0;0];
  end
  if (mod(ti,param_int)==0) % If true, update parameters
    if (abs(V(2)-V(1)) > p_tol)
      sigma_DA = sigma_DA - mu_p*(V(1)-U(1))/(V(2)-V(1));
    end
    if (abs(V(1)) > p_tol)
      rho_DA   = rho_DA   - mu_p*(V(2)-U(2))/(V(1));
    end
    if (abs(V(3)) > p_tol)
      beta_DA  = beta_DA  + mu_p*(V(3)-U(3))/(V(3));
    end
    lorenz_DA=@(V,FC)[sigma_DA*(V(2)-V(1));V(1)*(rho_DA-V(3))-V(2);...
                      V(1)*V(2)-beta_DA*V(3)] - FC;
  end
  U = U + dt*lorenz(U,epsilon*sqrt(dt)*randn(3,1));
  V = V + dt*lorenz_DA(V,FC);
end
semilogy(t,e_sol); hold on; semilogy(t,e_sigma); semilogy(t,e_rho); 
semilogy(t,e_beta); xlabel('time'); ylabel('Error'); axis('tight');
legend('Solution Error','|\Delta\sigma|','|\Delta\rho|','|\Delta\beta|');
\end{lstlisting}

\begin{multicols}{2}

\noindent Elizabeth Carlson\\
{\footnotesize
Department of Mathematics\\
University of Victoria\\
Email: \url{elizabeth.carlson@huskers.unl.edu}}\\[.2cm]
\\

\noindent Joshua Hudson\\ 
{\footnotesize
Combustion Research Facility\\
Sandia National Laboratories\\
Email: \url{jlhudso@sandia.gov}}\\[.2cm]
\\

\noindent Adam Larios\\ 
{\footnotesize
Department of Mathematics\\
University of Nebraska--Lincoln\\
Email: \url{alarios@unl.edu}}\\[.2cm]

\columnbreak 

\noindent Vincent R. Martinez\\
{\footnotesize
Department of Mathematics \& Statistics\\
CUNY Hunter College\\
Email: \url{vrmartinez@hunter.cuny.edu}}\\[.2cm]
 \\

\noindent Eunice Ng\\ 
{\footnotesize
Department of Mathematics\\
Stony Brook University\\
Email: \url{eunice.ng@stonybrook.edu}}\\[.2cm]

\noindent Jared P. Whitehead\\ 
{\footnotesize
Department of Mathematics \& Statistics\\
Brigham Young University\\
Email: \url{whitehead@mathematics.byu.edu}}\\[.2cm]

\end{multicols}

\end{document}